    \documentclass[12pt]{article}
    \usepackage{setspace}
    \usepackage{mathtools} % Upgrade from amsmath
      \mathtoolsset{showonlyrefs} % Only show eqn numbers if referenced
    \usepackage{amsthm, amssymb, eucal}
    \usepackage{upgreek}   % So we can use \Upsigma rather than \Sigma
    \usepackage{mathrsfs}  % \mathscr font for R and S in Section 5
    \usepackage{microtype} % Makes pdf output look better
    \usepackage[percent]{overpic} % Needed for diagram lables!
    \usepackage{enumerate} % Needed for special Item types
    \usepackage{float}     % to have "H" placement option for figures
    \usepackage[all]{xy}   % Needed for diagrams!
  \usepackage{geometry}
    \theoremstyle{plain}
    \newtheorem{thm}{Theorem}
    \newtheorem{prop}[thm]{Proposition}
  \DeclareMathOperator{\sym}{Isom}
  \newcommand{\set}[1]{\{ #1 \}}           % Usage: \set{x_1, x_2, x_3}
  \newcommand{\setofall}{\{\,}
  \newcommand{\suchthat}{\mid}
  \newcommand{\setend}{\,\}}
  \newcommand{\iso}{\cong}                 % Isomorphic.
  \newcommand{\gp}[1]{\left \langle #1 \right \rangle}  % for "group generated by"
  \DeclareMathOperator{\aut}{Aut}
  \DeclareMathOperator{\stab}{Stab}
  \newcommand{\abs}[1]{\left| #1 \right|}  % Single bar abs value symbol
  \newcommand{\numsys}[1]{\mathbb{#1}}     % Font for number system commands
  \newcommand{\zmod}[1]{\ensuremath{\numsys{Z}_{#1}}}
  \newcommand{\R}{\ensuremath{\numsys{R}}}
  \DeclareMathOperator{\fix}{Fix}
  \newcommand{\cube}{\mathcal{C}}
  \newcommand{\T}{\mathcal{T}}             % For spanning tree T
  \newcommand{\st}{G_{\mathcal{T}}}        % For stabilizer of T
\usepackage{color}

\usepackage[normalem]{ulem}  % For \sout editing command
\begin{document}
\onehalfspacing
\begin{center}
\Large
Unfoldings of the Cube
\end{center}

\begin{flushright}
Richard Goldstone\\
Manhattan College\\
Riverdale, NY 10471\\
\verb+richard.goldstone@manhattan.edu+

\vspace{2 mm}

Robert Suzzi Valli\\
Manhattan College\\
Riverdale, NY 10471\\
\verb+robert.suzzivalli@manhattan.edu+
\end{flushright}

\begin{abstract}
Just how many different connected shapes result from slicing a cube along some of its edges and unfolding it into the plane?  In this article we answer this question by viewing the cube both as a surface and as a graph of vertices and edges.  This dual perspective invites an interplay of geometric, algebraic, and combinatorial techniques.  The initial observation is that a cutting pattern which unfolds the cubical surface corresponds to a spanning tree of the cube graph.  The Matrix-Tree theorem can be used to calculate the number of spanning trees in a connected graph, and thus allows us to compute the number of ways to unfold the cube.  Since two or more spanning trees may yield the same unfolding shape,  Burnside's lemma is required to count the number of incongruent unfoldings. Such a count can be an arduous task.  Here we employ a combination of elementary algebraic and geometric techniques to bring  the problem within the range of simple hand calculations.
\end{abstract}

\subsection*{Introduction}

The question of how many different shapes result from slicing a cube along some of its edges and
unfolding it into the plane arose for us during an analysis of shortest paths on the cube.  When the second author found a way to count possibilities that brought most of the problem within the range of simple hand calculations, we felt that the analysis, with its interplay of geometric, algebraic, and combinatorial techniques, was worthy of its own presentation. A second paper \cite{goldsuzz} considers the actual shortest path(s) that can be obtained.

Here is an outline of the argument.  If we think of the vertices and edges of the cube as a graph, then edge-cutting patterns that produce unfoldings are precisely spanning trees of the cube graph.  A result called the Matrix-Tree theorem provides a computation for the number of spanning trees in a graph.  From this we get the total possible number of cutting patterns and so the total possible number of  unfoldings. 

This is just the beginning of the counting problem, because indistinguishable unfoldings are  obtained from cutting patterns that are congruent under an isometry of the cube.    To count the groups of indistinguishable cutting patterns, we have assumed familiarity with the basic language of  a group~$G$ acting on a set~$X$ and two elementary results, the orbit-stabilizer theorem for finite groups
\begin{equation}
 \abs{x^G}= \frac{\abs{G}}{\abs{\stab_G(x)}},
\label{orbstabthm}
\end{equation}
and Burnside's lemma
\begin{equation} \label{e:bl0}
 \text{\# orbits in $X$} = \frac{1}{\abs{G}} \sum_{g \in G} \abs{\fix(g)}.
\end{equation}
In these formulas, we use the notation from~\cite{jGalC06},
\[
 \stab_G(x)=\setofall g \in G  \suchthat x^g=x \setend \le G, \text{ all the elements of~$G$ that fix a particular~$x$}
\]
(``Stab'' here for ``stabilizer'').
 More generally, for any subset $A \subseteq X$,
\[
 \stab_G(A)=\setofall g \in G  \suchthat a^g \in A \text{ for all $a \in A$} \setend \le G.
\]
Note the requirement~$a^g \in A$.  We do not require $a^g=a$.  We also have
\[
\fix(g) = \setofall x \in X \suchthat x^g=x  \setend \subseteq X, \text{ all the elements of~$X$ that are fixed by a particular~$g$.}
\]

We alert the reader to a possibly unfamiliar notation: it is fairly common in permutation group theory to write functions on the right as exponents (see, for example, \cite{jDbMP96}), so~$x^f$ rather than~$f(x)$, and we have adopted that convention here.   That notation is extended in the orbit-stabilizer theorem statement, in which
\[
 x^G=\setofall x^g \suchthat g \in G \setend = \text{the orbit of~$x$ under the $G$-action.}
\]  

Burnside's lemma requires us to find, for each isometry of the cube, the number of spanning trees that are  invariant under that isometry.  This is where the rubber hits the road, since counting invariant spanning trees for all~48 isometries is still a big task.  We manage to cut that job down to reasonable size by finding many isometries to exclude from the party because they don't have what it takes to have invariant spanning trees.

The unfolding count  is not new and has been found by various authors in various ways. \cite{fBmP98, mJ1975, pT84}  Many authors have arrived at the strategy of using Burnside's lemma to count orbits of the cube isometry group acting on the set of spanning trees of the cube graph, the differences being in how the the sum of fixed points is analyzed. The published works we have been able to find have broader goals than just the analysis of the cube situation. With the aim of greater generality in the either the type and/or the dimension of the polyhedron, other authors have relied on more advanced combinatorial and geometric techniques and have been less interested in supplying specific details for the case of the cube. For this reason, the present work is more elementary and so is a good introduction to further results on unfolding in the literature.

\subsection*{Unfolding the cube}
We let~$\cube$ denote the cube, but confess to thinking about~$\cube$ in two different ways without making corresponding notational distinctions.  Sometimes,~$\cube$ denotes the cubical surface in~$\R^3$.  In this case, we speak of isometries of~$\cube$.   At other times,~$\cube$ denotes the graph formed by the vertices and edges of the cube.  In this case we speak of automorphisms, rather than isometries, of~$\cube$.  The context in which~$\cube$ appears should always make it clear which view of~$\cube$ is operative.   

As mentioned in the introduction, we will need the graph-theoretic notion of a \emph{spanning tree} of a graph~$\mathcal{G}$. A spanning tree is a subgraph that contains all the vertices of~$\mathcal{G}$ (this is the spanning part), has no circuits, and is connected (this is the tree part).

If a collection of squares in the plane has the property that whenever two squares intersect, the intersection is precisely an edge of both squares, we shall say the squares are \emph{joined along edges}.  If the boundary of the collection has a single component, we shall call the collection \emph{connected}. If a connected planar arrangement of six squares, joined along edges, folds up into a cube, we shall call the arrangement an \emph{unfolding}.  The most common term in the literature for an unfolding is a  \emph{net}. \cite{adA58,fBmP98,gcS75} However, we find ``net'' to be lacking in descriptive content and possibly a mistranslation of the German \emph{netz,} and so prefer ``unfolding.''

We may view the unfolding as having been created by cutting along
some of the edges of the cube, and this is where spanning trees come in.
 
\begin{prop}
\label{prop1}
In order to get an unfolding of the cube, the edge cutting pattern~$\T$ must be a spanning tree of~$\cube$.
\end{prop}

\begin{proof}
The boundary of the unfolding folds up into the cutting pattern~$\T$ on the cube.  Since the boundary is connected,~$\T$ must be a connected subgraph of~$\cube$ without circuits. Since every vertex of the cube must be represented by at least one vertex of the unfolding, $\T$ must be a spanning subgraph of~$\cube$,  and so~$\T$ is a spanning tree of~$\cube$.
\end{proof}

Figure~\ref{f:lcubnet} depicts an unfolding and the corresponding spanning tree in~$\cube$.  In order to label the faces clearly,~$\cube$ is depicted as viewed ``through'' the ``top'' face~\textbf{2}, which consequently does not appear as a labeled face in the diagram.  The edges of~$\cube$ drawn with double lines are the edges of the corresponding spanning tree.

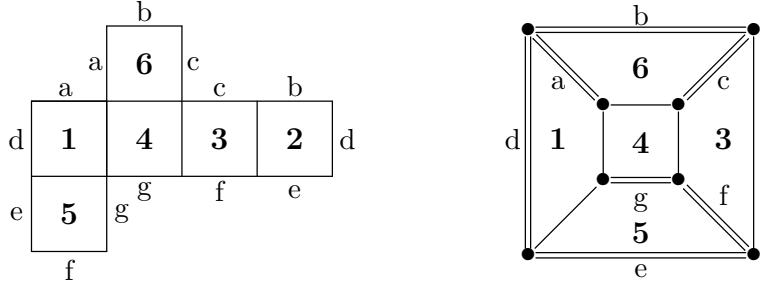
\begin{figure}[ht]      % immediate placement first, else top of next page.
\[
\begin{xy}
   (0,0)*{\begin{xy}(0,0)*{}="A"; (10,0)*{}="B"; (10,10)*{}="C";  (20,10)*{}="D"; (30,10)*{}="E"; (40,10)*{}="F";
         (40,20)*{}="G"; (30,20)*{}="H"; (20,20)*{}="I"; (20,30)*{}="J"; (10,30)*{}="K"; (10,20)*{}="L";
         (0,20)*{}="M"; (0,10)*{}="N";
         {\ar @{-} "A" ; "B"}; {\ar @{-} "B" ; "C"}; {\ar @{-} "N" ; "F"}; {\ar @{-} "F" ; "G"}; {\ar @{-} "G" ; "M"};
         {\ar @{-} "D" ; "J"}; {\ar @{-} "J" ; "K"}; {\ar @{-} "K" ; "C"}; {\ar @{-} "L" ; "M"}; {\ar @{-} "M" ; "A"};
         {\ar @{-} "L" ; "C"}; {\ar @{-} "E" ; "H"};
         (8.5,25)*{\text{\small a}}; (4.5, 21.5)*{\text{\small a}}; (15,32)*{\text{\small b}};
         (21.5,25)*{\text{\small c}}; (25,21.5)*{\text{\small c}}; (35,22)*{\text{\small b}}; (42,15)*{\text{\small d}};
         (35,8)*{\text{\small e}}; (25,8)*{\text{\small f}}; (15,8)*{\text{\small g}}; (12,5)*{\text{\small g}};
         (5,-2.5)*{\text{\small f}}; (-2,5)*{\text{\small e}}; (-2,15)*{\text{\small d}};
         (5,5)*{\textbf{5}}; (5,15)*{\textbf{1}}; (15,15)*{\textbf{4}}; (25,15)*{\textbf{3}}; (35,15)*{\textbf{2}}; (15,25)*{\textbf{6}};
         \end{xy}};
   (60,0)*{\begin{xy}
       (0,0)*{\bullet}="A"; (30,0)*{\bullet}="B"; (30,30)*{\bullet}="C"; (0,30)*{\bullet}="D";
       (10,10)*{\bullet}="E"; (20,10)*{\bullet}="F"; (20,20)*{\bullet}="G"; (10,20)*{\bullet}="H";
       {\ar @{=} "A"; "B"}; {\ar @{-} "B"; "C"}; {\ar @{=} "C"; "D"}; {\ar @{=} "D"; "A"};
       {\ar @{=} "E"; "F"}; {\ar @{-} "F"; "G"}; {\ar @{-} "G"; "H"}; {\ar @{-} "H"; "E"};
       {\ar @{-} "A"; "E"}; {\ar @{=} "B"; "F"}; {\ar @{=} "C"; "G"}; {\ar @{=} "D"; "H"};
       (15,3)*{\textbf{5}}; (26,15.5)*{\textbf{3}}; (15,25)*{\textbf{6}}; (4,15.5)*{\textbf{1}}; (15,15)*{\textbf{4}};
       (15,-2)*{\text{\small e}};  (15,32)*{\text{\small b}}; (26,8)*{\text{\small f}}; (-2,15.5)*{\text{\small d}};
       (26,23)*{\text{\small c}}; (4,23)*{\text{\small a}}; (15,7)*{\text{\small g}}
       \end{xy}}
\end{xy}
\]
 \caption{\label{f:lcubnet}An  unfolding of the cube and its spanning tree in $\cube$}  % Fill in name of label and caption under figure.
\end{figure}

The import of Proposition~\ref{prop1} is that the number of different ways to cut open the cube and unfold it into the plane is equal to the number of spanning trees of~$\cube$.  The following theorem provides a general approach to counting the spanning trees of a connected graph. (For a proof, see, for example, \cite{dmCmDhS79} or  \cite{jhvLrmW94}.) 

\begin{thm}[Matrix-Tree Theorem]
Let~$\mathcal{G}$ be a connected graph on $n$ vertices.  Let~$A$ be the adjacency matrix of~$\mathcal{G}$ and let~$D$ be a diagonal matrix whose diagonal contains the degrees of the corresponding vertices of~$\mathcal{G}$. Then the number of spanning trees of~$\mathcal{G}$ is the determinant of any~$(n-1) \times (n-1)$ principal submatrix of~$D-A$.
\end{thm}

We used a computer algebra system to apply the Matrix-Tree theorem with~$A$ the~$8 \times 8$ adjacency matrix of the cube graph (see Figure~\ref{f:lcubnet}),~$I$ the $8 \times 8$ identity matrix,  and~$D=3I$ since every vertex of the cube graph is of degree~3.  The result is that \emph{there are 384 ways to unfold the cube.}

The fact that there are~384 cutting patterns does not mean that there will be~384 unfoldings that are geometrically incongruent in the plane. In particular, if an isometry of the cube carries one cutting pattern to another, then the unfoldings must be congruent. Thus, there cannot be more unfolded shapes then there are orbits of cutting patterns under the action of~$\sym(\cube)$, the full isometry group of the cube.  More precisely, let~$\Upsigma$ be the set of all spanning trees of~$\cube$ and let~$G=\sym(\cube)$. Then~$\Upsigma$ has~384 elements, and~$G$ acts on~$\Upsigma$.  The number of orbits in~$\Upsigma$ under this action is an upper bound for the number of distinct unfoldings of~$\cube$.

The standard tool for computing the number of orbits under a group action is Burnside's lemma~\eqref{e:bl0}, which in the present context says
\begin{equation} \label{e:bl}
 \text{\# orbits in $\Upsigma$} = \frac{1}{\abs{G}} \sum_{g \in G} \abs{\fix(g)}.
\end{equation}
To use Burnside's lemma, we have to count, for each  of the~48 isometries of the cube, how many of the 384 spanning trees of the cube are invariant.

The next section lays the groundwork for this count by recalling some of the basic properties of the isometry group of the cube, and the section after that establishes facts that enable us to perform this count relatively easily by hand.

\subsection*{Overview of cube symmetry}
It is well-known that~$G=\sym(\cube)$ is isomorphic to~$S_4 \times \zmod{2}$ and so is of order~48.  See, for example, Chapters~8 and~10 of \cite{maA88}. The direct factor~$S_4$ is isomorphic to the group of rotations of the cube, a result that depends on the observation that any permutation of the four space diagonals of the cube can be carried out by rotations. 

There are three types of non-identity rotations  in~$G $:
\begin{enumerate}[\indent (Rot 1)]
\item 
Rotations through an axis perpendicular to a pair of opposite faces of~$\cube$ and through the centers of those faces.  The angles of rotation are multiples of $90^{\circ}$.
\item 
Rotations around an axis through the midpoints of a pair of opposite edges.  The angles of rotation are multiples of~$180^{\circ}$.
\item 
Rotations around one of the space diagonals of the cube.  The angles of rotation are multiples of~$120^{\circ}$.
\end{enumerate}

\begin{figure}[h]
\centering
\begin{overpic}[scale=0.3]{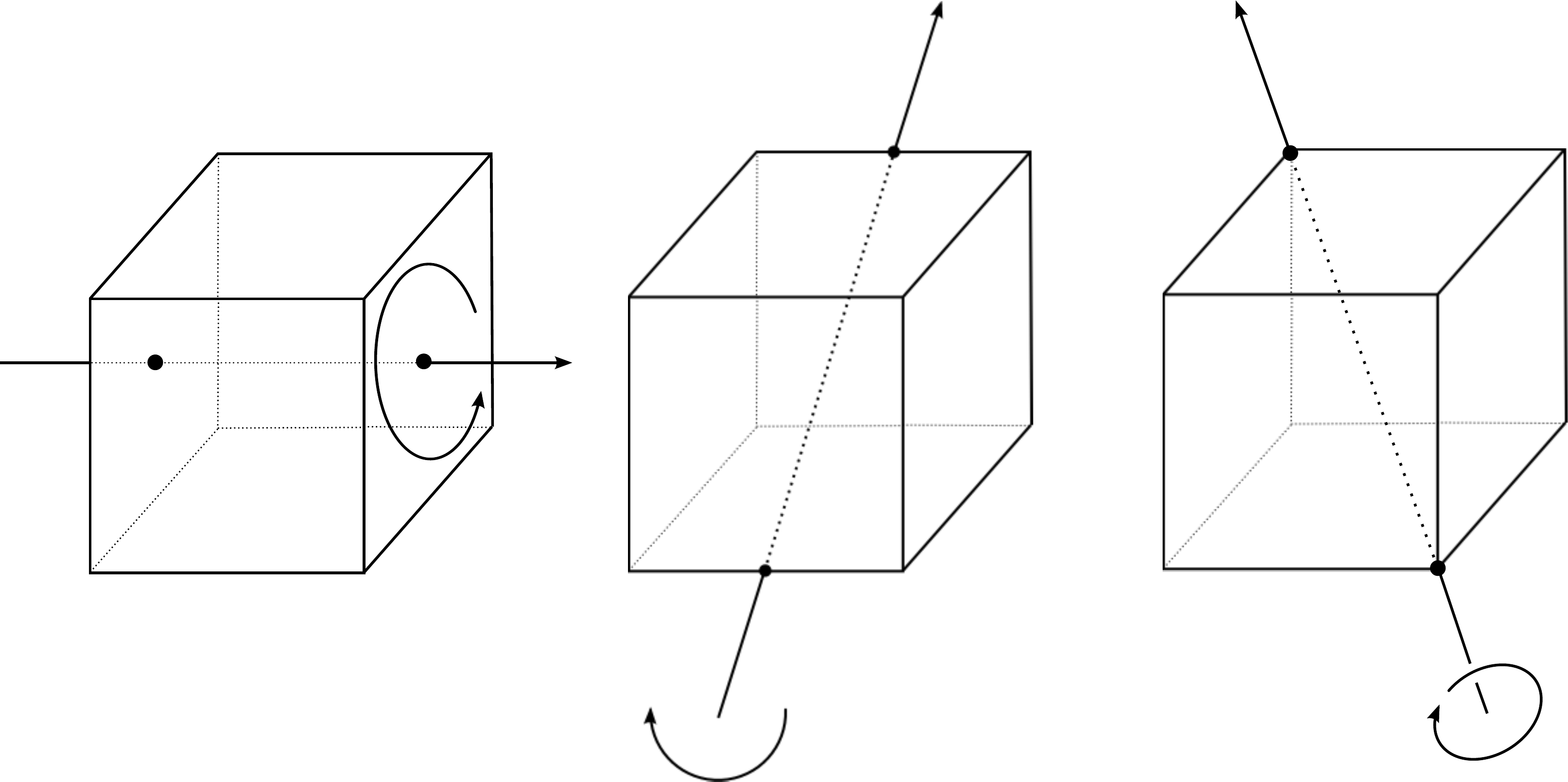}
% -------------------------------
\put(23,10.5){\footnotesize 1}
\put(32,21.5){\footnotesize 2}
\put(32,40.5){\footnotesize 3}
\put(21.5,31.5){\footnotesize 4}
\put(4,10.5){\footnotesize 5}
\put(13.5,21.5){\footnotesize 6}
\put(13,40.5){\footnotesize 7}
\put(3.5,30){\footnotesize 8}
% -------------------------------
\put(38.5,10.5){\footnotesize 1}
\put(57,10.5){\footnotesize 2}
\put(56.3,31.5){\footnotesize 3}
\put(38,30){\footnotesize 4}
\put(47.5,21.5){\footnotesize 5}
\put(66.5,21.5){\footnotesize 6}
\put(66,40.5){\footnotesize 7}
\put(47.5,40.5){\footnotesize 8}
% -------------------------------
\put(73,10.5){\footnotesize 1}
\put(93,12){\footnotesize 2}
\put(90.5,31.5){\footnotesize 3}
\put(72.5,30){\footnotesize 4}
\put(81.5,21.5){\footnotesize 5}
\put(100.5,21.5){\footnotesize 6}
\put(100.5,40.5){\footnotesize 7}
\put(82.5,41){\footnotesize 8}
% -------------------------------
\end{overpic}
\caption{Cube Rotation Isometry Types Rot~1, Rot~2, and Rot~3.} \label{f:curots}
\end{figure}
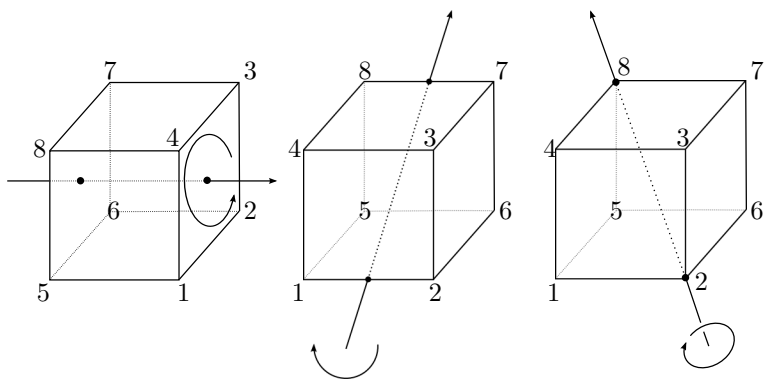

The table below gives typical cycle structure for each rotation type and the corresponding cycle structure for the representation in~$S_4$.

\[
\begin{array}{llcc} \hline\noalign{\smallskip}
\text{ Rotation Type} &\text{Rotations} \subset S_8 & \text{Effect on Space Diagonals} & \text{Rotations}=S_ 4 \\ \noalign{\smallskip}\hline\noalign{\medskip}
\text{(Rot 1)}\quad 90^\circ
& 
\rho_1=(1234)(5678) 
& 
\begin{pmatrix}  
[17] & [28] & [35] & [46]\\[0.3em]
[28] & [35] & [46] & [17]
\end{pmatrix}
&
(1234)
\\[1.5em]
\text{(Rot 1)} \quad 180^\circ
& 
\rho_1^2=(13)(24)(57)(68) 
& 
\begin{pmatrix}  
[17] & [28] & [35] & [46]\\[0.3em]
[35] & [46] & [17] & [28]
\end{pmatrix}
&
(13)(24)
\\[1.5em]
\text{(Rot 2)} \quad 180^\circ
&
\rho_2=(12)(35)(46)(78) 
&
\begin{pmatrix}  
[17] & [28] & [35] & [46]\\[0.3em]
[28] & [17] & [35] & [46]
\end{pmatrix}
&
(12)
\\[1.5em]
\text{(Rot 3)} \quad 120^\circ
&
\rho_3=(136)(475)
& 
\begin{pmatrix}  
[17] & [28] & [35] & [46]\\[0.3em]
[35] & [28] & [46] & [17]
\end{pmatrix}
&
(134) \\ \noalign{\medskip}\hline
 \end{array}
\]

\medskip
We view the copy of~$\zmod{2}$ appearing in~$G \iso S_4 \times \zmod{2}$ as being generated by a map~$\alpha$ that we call the \emph{antipodal} map.  The map~$\alpha$ interchanges the endpoints of  each space diagonal of~$\cube$. Since the diagonals are unoriented for the purposes of representing the rotations of the cube, it is clear that~$\alpha$ commutes with all the elements of~$S_4$.  Composing~$\alpha=(17)(28)(35)(46)$ with the rotations described above gives the following table of results:

\[
\begin{array}{llll} \hline\noalign{\smallskip}
\text{Rotation Type}&&  \text{Rotation} \circ \alpha & \text{Isometry Type}  \\ \noalign{\smallskip}\hline\noalign{\medskip}
\text{(Rot 1)}\quad 90^\circ & \rho_1
& 
\rho_1 \alpha = (1234)(5678)\alpha = (1836)(2547)
&
\text{(Rot 1)} \circ \text{(Ref 2)}

\\[1.5em]
\text{(Rot 1)} \quad 180^\circ & \rho_1^2
& 
\rho_1^2\alpha=(13)(24)(57)(68) \alpha=(15)(26)(37)(48)
&
\text{(Ref 2)}

\\[1.5em]
\text{(Rot 2)} \quad 180^\circ & \rho_2
&
\rho_2\alpha=(12)(35)(46)(78) \alpha= (18)(27)
&
\text{(Ref 1)}

\\[1.5em]
\text{(Rot 3)} \quad 120^\circ & \rho_3
&
\rho_3\alpha=(136)(475) \alpha=(156734)(28)
& 
\text{(Rot 1)} \circ \text{(Ref 1)} \\ \noalign{\medskip}\hline

 \end{array}
\]

\medskip\noindent
We find that among the resulting isometries are two types of reflections (see Figure~\ref{f:curef}):
\begin{enumerate}[\indent (Ref 1)]
\item 
A reflection in a plane containing a pair of opposite edges.
\item 
A reflection in a plane through the midpoints of a set of four parallel edges.
\end{enumerate}
We mention these reflections because they have a role in counting arguments that follow.

\begin{figure}[h]
\centering
\begin{overpic}[scale=0.3]{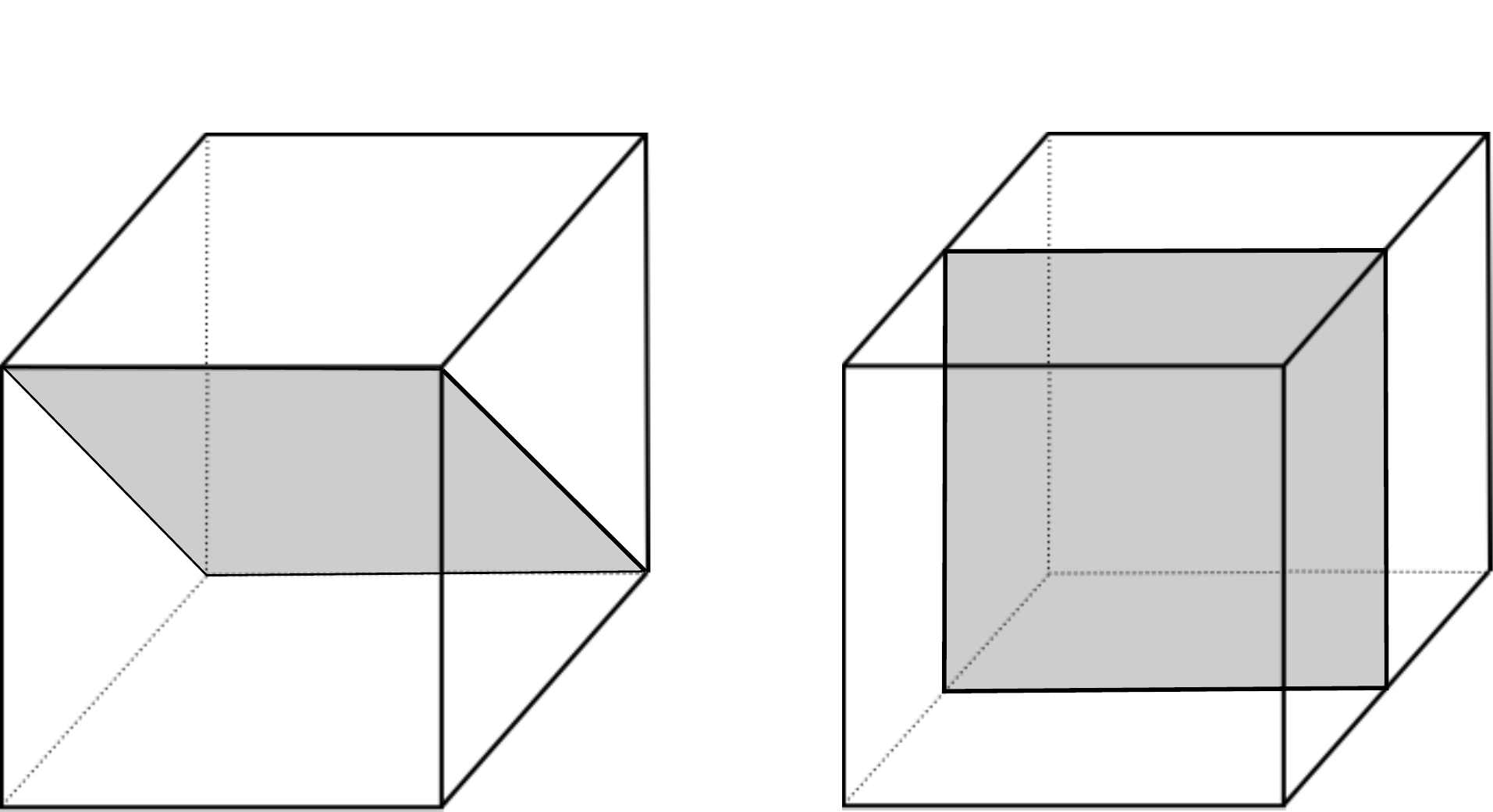}
\put(-3,-3){\footnotesize 1}
\put(31,-3){\footnotesize 2}
\put(27.5,31){\footnotesize 3}
\put(-3,31){\footnotesize 4}
\put(11.5,11.5){\footnotesize 5}
\put(43,11.5){\footnotesize 6}
\put(43.5,45){\footnotesize 7}
\put(10,45){\footnotesize 8}
% ------------------------------
\put(54,-3){\footnotesize 1}
\put(87,-3){\footnotesize 2}
\put(83.5,31){\footnotesize 3}
\put(53,31){\footnotesize 4}
\put(69,14){\footnotesize 5}
\put(99.5,11.5){\footnotesize 6}
\put(100,45){\footnotesize 7}
\put(66,45){\footnotesize 8}
\end{overpic}
\caption{Cube Reflection Isometry Types Ref~1 and Ref~2.} \label{f:curef}
\end{figure}
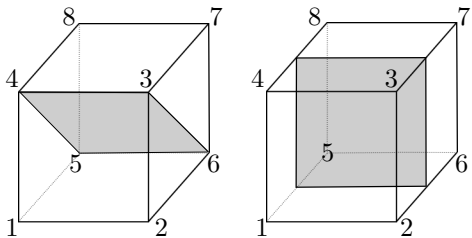

There are two properties of the group of isometries of the cube that we will need in order to analyze which cube isometries have invariant spanning trees.  The first property is about possible orders of certain elements.  Since~$G$ is of order~48, among the elements whose order is a multiple of~3 could, in principle, be elements of order~3,~6,~12,~24, and~48. In fact, since $G \cong S_4 \times \zmod{2}$, we can show that only the first two of these occur:

\begin{prop} \label{l:osix}
The elements of~$G$ whose order is a multiple of~3 are either of order~3 or of order~6.
\end{prop}

\begin{proof}
The possible types of disjoint cycles in~$S_4$ are~$(1\;2)$ and~$(1\;2)(3\;4)$, both of order~2,~$(1\;2\;3)$ of order~3, and~$(1\;2\;3\;4)$ of order~4. So the only elements of order a multiple of~3 in~$S_4$ are the 3-cycles, and these have order~3. The direct product decomposition $G \cong S_4 \times \zmod{2}$ means that the elements of $G$ can be written uniquely as either $\rho$ or $\rho\alpha$ for some $\rho \in S_4$ and $\alpha \in \zmod{2}$, and we have just seen that~$\rho$ can only be a 3-cycle. Consequently, the elements of order a multiple of~3 in~$G$ are either of order~3 (if of form $\rho$) or of order~6 (if of form $\rho\alpha$).
\end{proof}

The second property of the cube isometry group we will need concerns edge stabilizers. The group~$G$ operates transitively on the set of edges of~$\cube$, meaning that given any pair of edges, there is at least one isometry that carries one edge to  the position of the other edge.  In fact, the face-centered rotations by themselves are enough to move any edge to, say, the position of the bottom edge in the front face of~$\cube$. For any edge~$e$ of~$\cube$, recall that~$\stab_G(e)$ denotes the subgroup of~$G$ that carries~$e$ to itself.  Note that since we view the edges of~$\cube$ as unoriented,  group elements that interchange the endpoints of the edge~$e$ belong to~$\stab_G(e)$.

\begin{prop} \label{l:stab4}
$\abs{\stab_G(e)}=4$.
\end{prop}

\begin{proof}
According to the orbit-stabilizer theorem~\eqref{orbstabthm}, the orbit $e^G$ of any edge $e$ under the action of $G$ must satisfy
\[
\abs{e^G}= \frac{\abs{G}}{\abs{\stab_{G}(e)}}.
\]
The fact that the action of~$G$ on the set of edges of~$\cube$ is transitive means that the entire set of~12 edges is a single orbit, that is~$\bigl|e^G \bigr| =12$.  Since~$\abs{G}=48$, any edge-stabilizer subgroup must be of order~$48/12=4$.%
% Footnote %-----------------------------------------------------------------------
\footnote{Alternatively, we can see directly that the edge-stablizer is of order~4 by  noting, with the aid of Proposition~\ref{t:3sym}, that $\stab_G(e) \cong \zmod{2} \times \zmod{2}$.  The two $\zmod{2}$-generators are of the types Rot~2 and~Ref~2.}
% End Footnote % ------------------------------------------------------------------
\end{proof}

\subsection*{Isometries of the cube with invariant spanning trees}
We are now ready to begin the process of computing the size of the
sets~$\fix(g)$,~$g \in G$, for the action of~$G=\sym(\cube)$ on the set~$\Upsigma$ of spanning trees of~$\cube$.   A spanning tree~$\T$ is invariant under~$g \in G$ if $\T^g=\T$ (recall that we are using exponential notation for functions); individual edges and vertices may be moved as long as $g$ carries $\T$ to itself.   An analogous statement holds for a single edge invariant under~$g$: the edge can either be unmoved or reversed by the action.  This situation leads to some potentially confusing terminology about invariant edges.  If we think of the edges as members of a set~$E$ whose elements are permuted by the action of~$g$, then an edge~$e\in E$ that is either unmoved or reversed geometrically is simply a \emph{fixed point} for the $g$-action on~$E$, because from the set-element perspective $e^g=e$ regardless of whether the edge is unmoved or reversed.  So an edge that is fixed by permutations of~$E$ may either be fixed or reversed by the corresponding isometries of~$\cube$, and this is why we speak of invariant edges in the geometric context.

Returning now to the question of cube isometries with invariant spanning trees, our first observation is that if a cube isometry wants to have an invariant spanning tree, it can't just juggle the edges of that tree in any old way, it has to play nice by keeping at least one edge of the tree invariant.

\begin{prop} \label{t:fe}
Let $g \in G$ and suppose $\T$ is a spanning tree of $\cube$  invariant under $g$.  Then there is at least one edge of~$\T$ that is invariant under $g$.
\end{prop}

\begin{proof}
Let~$\T$ be a spanning tree of~$\cube$. The cube, and hence the spanning tree~$\T$, has eight vertices. Since a tree with~$n$ vertices must have~$n-1$ edges,~$\T$ has seven edges. Let~$E$ denote this set of edges.  The subgroup of elements of~$G$ for which~$\T$ is an invariant spanning tree is~$\st=\stab_G(\T)$, and~$\st$ also acts on~$E$.  Since~$E$ has seven elements, this action amounts to a homomorphism (neither 1-1 nor onto) $\gamma:\st \to S_7$.  If the edges of the cube are labeled from~1 to~12 in such a way that the edges in~$\T$ are labeled~1 to~7, then the homomorphism~$\gamma$ simply ``forgets'' all cycles in any~$h \in \st$ involving the labels~8 through~12. This means that~$\st$ and its image~$(\st)^\gamma < S_7$ have exactly the same orbits in~$E$, the advantage of using~$S_7$ being that for any $h \in \st$, the orbits of~$h$ acting on~$E$ are precisely the cycles in the disjoint cycle representation of~$h^\gamma \in S_7$.

Some of the properties we've found for~$G$ also apply to~$(\st)^\gamma < S_7$.  In particular, since the order of a homomorphic image of an element must divide the order of the element, every $h^\gamma \in S_7$ has order dividing 48, and an element of~$(\st)^\gamma$ whose order is a multiple of~$3$ can only have order~3 or~6.

Our claim is that no element of~$\st$ can act on~$E$ without a fixed point.  To prove this, suppose we have an   $h \in \st$ fixing no elements of~$E$. Then~$h^\gamma \in S_7$ fixes no elements of~$E$.  In the disjoint cycle decomposition of~$h^\gamma$,  all cycles of~$h^\gamma$ would have to be of length~2 or greater, and each of the seven elements of~$E$ would have to appear in exactly one cycle of~$h^\gamma$.  Consider the possibilities:

\begin{enumerate}
\item
The element~$h^\gamma$ is a 7-cycle. This is impossible since the order of~$h^\gamma$ would be~7 and~$7 \nmid 48$. 
\item 
The element $h^\gamma$ is the  product of a 5-cycle and disjoint  2-cycle. This would mean that the order of $h^\gamma$ is~10, which is impossible since $10 \nmid 48$.
\item 
The element $h^\gamma$ is the  product of a  4-cycle and disjoint  3-cycle. This would mean that the order of $h^\gamma$ is~12, which is impossible since, as a consequence of Proposition~\ref{l:osix}, the largest  order in~$(\st)^\gamma$ that is a multiple of~3 is~6.
\item \label{i:tttc} 
The element~$h^\gamma$ is the disjoint product of two transpositions and a 3-cycle. This would mean that the order of~$h^\gamma$ is~6, which means that the order of~$h$ is divisible by~6.  In view of Proposition~\ref{l:osix}, this means that the order of~$h$ must also be~6.  But this cannot be, since for any $e \in E$, the orbit-stabilizer theorem  says that

\[
 \bigl | e^{\gp{h}} \bigr | = \frac{\abs{h}}{\abs{\stab_{\gp{h}}(e)}} = \frac{\abs{h}}{\abs{\stab_G(e) \cap \gp{h}}} = \frac{6}{\abs{\stab_G(e) \cap \gp{h}}}.
\]
In light of Proposition~\ref{l:stab4}, any common subgroup of~$\stab_G(e)$ and~$\gp{h}$ must have order dividing~4, so the only possible orders for~$\stab_G(e) \cap \gp{h}$ are~1,~2, and~4, with~4 excluded here because $4 \nmid 6$.  This means that the only possible orbit sizes for orbits of~$e$ under~$\gp{h}$ are~6 and~3. These are both impossible, not only because the assumption for this item requires two orbits of size~2, but also because the seven elements of $E$ cannot be partitioned into subsets of sizes~3 and~6.
\end{enumerate}
All the possibilities for an element~$h \in \stab_G(\T)$ to act on~$E$ without fixed points lead to contradictions, so it follows that at least one element of~$E$  must be fixed by~$h$.  
\end{proof}

It is interesting to note that this result, which is fundamental for everything that follows, is not really tightly controlled by either cube geometry or cube graph structure.  Rather, it is a generic algebraic result about group actions, namely if~$S_4 \times \zmod{2}$ acts transitively on a set of size~12, and if there is a subgroup~$H < S_4 \times \zmod{2}$ acting on a subset~$E$ of size~7, then each element of~$H$ must fix an element of~$E$. Referring back to the cube, the result would apply to any set of seven edges~$E$ that is invariant under a subgroup of the isometries of the cube, whether or not those seven edges belong to a spanning tree.

We have now seen that in order to have an invariant spanning tree, a cube isometry will have to have at least one invariant edge.  Not every isometry is up to the challenge of keeping an edge invariant, and only those dexterous isometries that can manage this task are candidates for having an invariant spanning tree.  So, for the moment, we leave the question of invariant spanning trees and ask which cube isometries have invariant edges.

\begin{prop} \label{t:3sym}
There are three types of isometries of~$\cube$ with an invariant edge: a~$180^{\circ}$ rotation around an axis joining the midpoints of a pair of opposite edges of~$\cube$ (type~\emph{Rot~2} in Figure~\ref{f:curots}), a reflection in a plane containing a pair of opposite edges of~$\cube$ (type~\emph{Ref~1} in Figure~\ref{f:curef}), and a reflection in a plane through the midpoints of a set of parallel edges of~$\cube$ (type~\emph{Ref~2} in Figure~\ref{f:curef}).
\end{prop}

\begin{proof}
In order for an edge~$e$ to be invariant under a non-trivial rotation, the axis of rotation must pass through the midpoint of the edge. The only such rotations are the~$180^{\circ}$ rotations whose axes pass through the midpoints of a pair of opposite edges.

If an isometry of~$\cube$ is not a rotation, it must be of the form~$\rho\alpha$, where~$\rho$ is a rotation and~$\alpha$ is the antipodal map. If~$e$ is an edge of~$\cube$ that is fixed by~$\rho\alpha$, then~$e^{\rho\alpha}=e$ is equivalent to~$e^{\rho}=e^{\alpha}$; the rotation~$\rho$ must carry the edge~$e$ to the edge~$e^{\alpha}$.  The edges~$e$ and~$e^{\alpha}$ are opposite each other in~$\cube$ and so are parallel, hence they determine a plane $\Pi$ as in type~{Ref 1} of Figure~\ref{f:curef}, and so~$\Pi$ must be carried to itself by the rotation~$\rho$.  There are two ways this can happen: either the axis of rotation of~$\rho$ is perpendicular to~$\Pi$ as in the left drawing in Figure~\ref{f:rotperp}, or the axis of rotation is contained in~$\Pi$ as in the middle drawing in Figure~\ref{f:rotperp}.  (There are, of course, infinitely many axes perpendicular to~$\Pi$ and embedded in~$\Pi$, but the illustrated ones are the only possibilities that also interchange~$e$ and~$e^\alpha$.) In both cases,  the angle of rotation of $\rho$ must be~$180^{\circ}$ in order
to carry~$e$ to~$e^{\alpha}$.

\begin{figure}[h]
\centering
\begin{overpic}[scale=0.3]{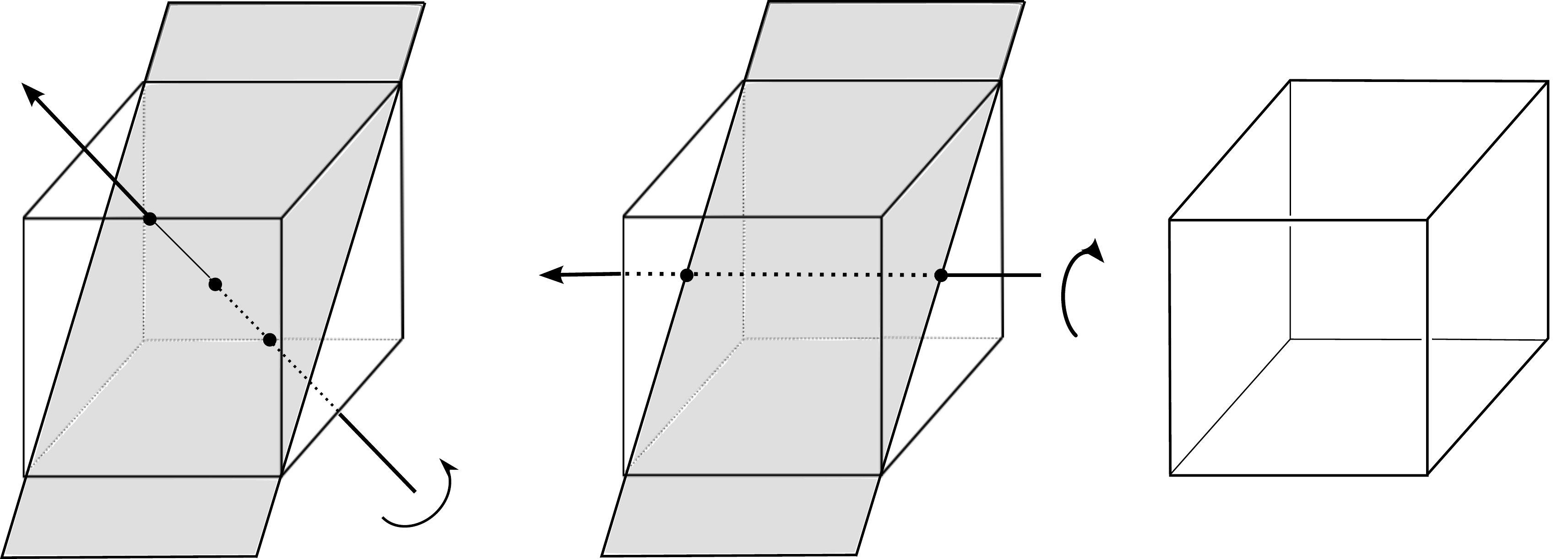}
\put(16.5,31.5){\small $e^\alpha$}
\put(55,31.5){\small $ e^\alpha$}
\put(13,24.5){$\Pi$}
\put(51,24.5){$\Pi$}
\put(8.5,3){$e$}
\put(47,3){$e$}
\put(72.5,3){\small 1}
\put(91.5,3){\small 2}
\put(92,20){\small 3}
\put(72,20){\small 4}
\put(83,15){\small 5}
\put(100,15){\small 6}
\put(99,32){\small 7}
\put(81,32){\small 8}
\end{overpic}
\caption{Rotation axes possibilities and vertex labels} \label{f:rotperp}
\end{figure}
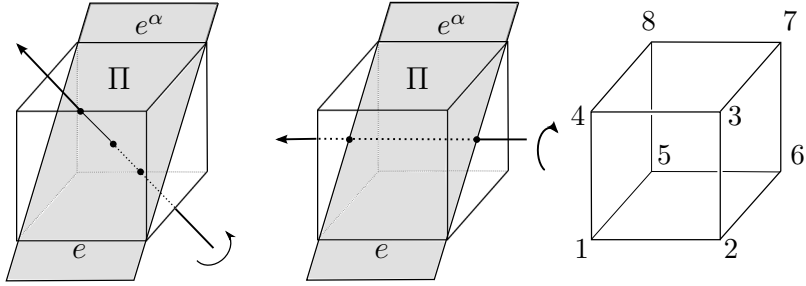

If the axis of rotation of~$\rho$ is perpendicular to~$\Pi$, then that axis  must pass through the midpoints of the pair of opposite edges of~$\cube$ that are parallel to~$\Pi$. Using the vertex labeling given in Figure~\ref{f:rotperp}, we can write the composition~$\rho \alpha$ as
\[
 \underbrace{(1\; 7)(2 \; 8)(3 \; 4)(5 \; 6)}_\rho \, 
 \underbrace{(1 \;7 )( 2\; 8)(3 \; 5)( 4\;6 )}_\alpha
=( 3\; 6)( 4\;5 ),
\]
which is a reflection in the plane~$\Pi$, a plane containing a pair of opposite edges of~$\cube$, the second of the possibilities claimed.   

If the axis of rotation of~$\rho$ is contained in~$\Pi$, then that axis must pass through the midpoints of a pair of opposite faces of~$\cube$.  Using the vertex labeling given in Figure~\ref{f:rotperp}, we can write the composition~$\rho \alpha$ as
\[
 \underbrace{(1 \; 8)(4 \; 5)(2 \; 7)(3 \; 6)}_\rho\,
 \underbrace{(1 \;7 )( 2\; 8)(3 \; 5)( 4\;6 )}_\alpha
 = (1 \; 2)(3\;4)(5\;6)( 7\; 8),
\]
which is a reflection in a plane perpendicular to the edges~$e$ and~$e^\alpha$ and passing through their midpoints, the third of the possibilities claimed.
\end{proof}

So far we've noted that an isometry of the cube with an invariant spanning tree must have an invariant edge.  We then drew back and considered all the cube isometries to see which ones leave and edge invariant, regardless of whether or not they have an invariant spanning tree.  We now take the isometry types we found with an invariant edge and check which, if any, of these types actually has an invariant spanning tree.

Proposition~\ref{t:3sym} provides three potential isometry types, and our first step is to eliminate one by showing that a reflection $\varphi$ in a plane determined by a pair of opposite edges of~$\cube$ (type~Ref 1 of Figure~\ref{f:curef}) has no invariant spanning tree.  To do this, suppose~$S$ is a connected spanning subgraph of~$\cube$ that is invariant under~$\varphi$.  We shall show that~$S$ must contain a circuit and so cannot be a tree. 

The subgraph~$S$ has four vertices fixed by~$\varphi$ and two pairs of vertices interchanged by~$\varphi$.  We can choose a vertex~$A$ from one fixed pair and a vertex~$B$ from the other fixed pair so that~$A$ and~$B$ are not adjacent in~$S$. Since~$S$ is connected, there is a path in~$S$ from~$A$ to~$B$, and since~$A$ and~$B$ are not adjacent, this path must pass through at least one vertex~$C$ that is moved by~$\varphi$. Since~$S$,~$A$, and~$B$ are all fixed by~$\varphi$, there is also a path in~$S$ from~$A$ to~$B$ through~$C^{\,\varphi} \ne C$. The union of these two distinct paths in~$S$ from~$A$ to~$B$ must contain a circuit in~$S$, and so~$S$ cannot be a tree. It follows that no spanning tree of~$\cube$ can be fixed by~$\varphi$.

We still have to show that the remaining rotation (Rot~2) and reflection (Ref 2) isometries  mentioned in Proposition~\ref{t:3sym} actually have some invariant spanning trees. The terminal entries of Figure~\ref{f:trees1} give examples of the spanning trees invariant under each type of isometry.  The top four entries are spanning trees invariant under one particular reflection, and the bottom eight entries are spanning trees invariant under one particular rotation. With these examples in hand, we have established the following result:

\begin{prop}\label{t:2sym}
There are two types of isometries of~$\cube$ with an invariant
spanning tree: a~$180^{\circ}$ rotation around an axis joining the
midpoints of a pair of opposite edges of~$\cube$ (type~\emph{Rot~2} in Figure~\ref{f:curots}) and a reflection in a plane through the midpoints of a set of parallel edges of~$\cube$ (type~\emph{Ref~2} of Figure~\ref{f:curef}).
\end{prop}

The two types of isometries of~$\cube$ identified in Proposition~\ref{t:2sym} each interchange the endpoints of  pairs of opposite edges of the cube.  But this is too much of a good thing for spanning tree invariance, and the next proposition verifies that no spanning tree of the cube can contain both these edges.

\begin{prop} \label{p:1edge}
A non-identity automorphism of a tree can reverse at most one edge.
\end{prop}

\begin{proof}
Let~$\T$ be a tree and suppose it has an automorphism~$g \in \aut(\T)$ that reverses the edge~$\set{A,A'}$.  If this edge is deleted from~$\T$, a graph~$\T^{\,-}$ consisting of two tree components~$C_A$ and~$C_{A'}$ results, and~$g \in \aut(\T^{\,-})$ as well. (The component~$C_A$ is the one containing~$A$;~$C_{A'}$ contains~$A'$.)  Since~$A^g=A'$, it follows that~$(C_A)^g=C_{A'}$. Any other edge~$\set{B,B'}$ must lie in either~$C_A$ or~$C_{A'}$, and so~$g$ must move~$\set{B,B'}$ from one component of~$\T^{\,-}$ to the other. But this means~$g$ cannot reverse~$\set{B,B'}$, which establishes the claim.
\end{proof}

\subsection*{Counting orbits in $\Upsigma$} \label{s:numorbs}

So here is where things stand: \emph{an isometry of the cube with an invariant spanning tree is either a~$180^\circ$ rotation around an axis joining the midpoints of a pair of opposite edges (\emph{Rot~2} in Figure~\ref{f:curots}) or a reflection in a plane through the midpoints of a set of parallel edges (\emph{Ref~2} in Figure~\ref{f:curef}), and the invariant spanning tree in question will have a single invariant edge that is reversed by the isometry.}  These observations make it feasible to calculate, by hand, the sum~\eqref{e:bl}
provided by Burnside's lemma.  

Let~$\mathscr{R}$ denote the set of all~$180^{\circ}$ rotations around an axis through the midpoints of a pair of opposite edges of~$\cube$, and let~$\mathscr{F}$ denote the set of all reflections in a plane through the midpoints of a set of parallel edges of~$\cube$.  Clearly, the elements in~$\mathscr{R}$  all have the same number of fixed points, as do the elements in~$\mathscr{F}$. (Both~$\mathscr{R}$ and~$\mathscr{F}$ are conjugacy classes in~$G$, but we don't have to use this fact here.) Hence, Burnside's lemma becomes
\begin{align}
 \text{\# orbits in $\Upsigma$} &=
   \frac{1}{48} \left( 384 + \sum_{\rho \in \mathscr{R}} \fix(\rho)
   + \sum_{\varphi \in \mathscr{F}} \fix(\varphi) \right),\\
   &= \frac{1}{48} \bigl( 384 + \abs{\mathscr{R}}\fix(\rho_0) + \abs{\mathscr{F}}\fix(\varphi_0) \bigr),\\
   &= \frac{1}{48} \bigl( 384 + 6 \fix(\rho_0) + 3\fix(\varphi_0) \bigr),\\
   &= 8 + \frac{\fix(\rho_0)}{8} + \frac{\fix(\varphi_0)}{16}, \label{e:blfinal}
\end{align}
where~$\rho_0 \in \mathscr{R}$ and~$\varphi_0 \in \mathscr{F}$ are particular single elements. The values~$\abs{\mathscr{R}}=6$ and~$\abs{\mathscr{F}}=3$ come from the fact that the elements of~$\mathscr{R}$, respectively~$\mathscr{F}$, are in 1-1 correspondence with an axis of rotation, respectively a plane of reflection.  To count the rotation axes for~$\mathscr{R}$, use the fact that each axis passes through the midpoints of a pair of opposite edges.  The cube has twelve edges and so six pairs of opposite edges, giving~$\abs{\mathscr{R}}=12/2=6$.  To count reflection planes for~$\mathscr{F}$, use the fact that each plane passes through the midpoints of four parallel edges, and there are three such groups of parallel edges, yielding~$\abs{\mathscr{F}}=12/4=3$. 

It remains to count the number of invariant spanning trees for the single elements~$\rho_0$ and~$\varphi_0$. Each of these isometries leaves invariant (by reversal) pairs of opposite edges of~$\cube$.  The rotation~$\rho_0$ leaves invariant a pair of opposite edges (the one whose midpoints are contained in the axis of rotation), and the reflection~$\varphi_0$ leaves invariant two pairs of opposite edges (the ones whose midpoints are contained in the plane of reflection).  We have established that each invariant spanning tree contains exactly one invariant edge.  So, for the rotation~$\rho_0$, we can choose one of its two possible invariant edges, find the number of invariant spanning trees that have the chosen edge invariant, and then double that count to get the number of spanning trees invariant  under~$\rho_0$.  For the reflection~$\varphi_0$, we can choose one of its four possible invariant edges, find the number of invariant spanning trees that have the chosen edge invariant, and then quadruple that count to get the number of spanning trees invariant under~$\varphi_0$.

Our strategy for enumerating spanning trees with a specified invariant edge relies on the observation in the proof of Proposition~\ref{p:1edge} that if a reversed edge is deleted, the result is a pair of isomorphic tree components, either one of which is the image of the other under the isometry that reverses the deleted edge.  This means that we can ``grow'' all possible invariant trees, starting at an endpoint of an edge that will ultimately be the reversed edge. Each time we add an edge to the growing tree, we employ the relevant isometry, either~$\rho_0$ or~$\varphi_0$, to produce a corresponding edge  and so maintain the required symmetry.  We use an orbit diagram to implement the tracing process simultaneously with the growing process.  At each stage of the growing process, an open circle denotes the endpoint or endpoints from which the next edge or pair of edges will ``sprout.''   The orbit labels indicate the edges that have to be added to the growing tree to ``balance'' the newly-sprouted edges. See Figure~\ref{f:trees1} below.

\begin{figure}[H]
\[\begin{xy}
 (-35, 136)*{\text{\small Edge orbits for $\varphi_0$}};
 (-32,152)*{\begin{xy}
        % vertices
          (8,24)*{}="T3"; (24,24)*{}="T2";
        (0,16)*{}="T0";  (16,16)*{}="T1";
          (8,8)*{}="B3"; (24,8)*{}="B2";
         (0,0)*{}="B0";  (16,0)*{}="B1";
        % edges
        {\ar @{-} "T0";"T1"}; {\ar @{-}|4 "T1";"T2"}; {\ar @{-} "T2";"T3"}; {\ar @{-}|4 "T3";"T0"}; % top
        {\ar @{-}|>>>>>{\strut 1} "B0";"T0"}; {\ar @{-}|>>>>>{\strut 1} "B1";"T1"}; {\ar @{-}|<<<<<{\strut 3} "B2";"T2"}; {\ar @{..}|<<<<<{\strut 3} "B3";"T3"}; % sides
        {\ar @{-} "B0";"B1"};  {\ar @{-}|2  "B1";"B2"}; {\ar @{..} "B2";"B3"}; {\ar @{..}|{\strut 2} "B3";"B0"}; % bottom
        \end{xy}};
 (0,152)*+{\begin{xy}
        % vertices
          (4,12)*{}="T3"; (12,12)*{}="T2";
        (0,8)*{}="T0";  (8,8)*{}="T1";
          (4,4)*{}="B3"; (12,4)*{}="B2";
         (0,0)*{}="B0";  (8,0)*{\circ}="B1";
        % edges
        {\ar @{..} "T0";"T1"}; {\ar @{..} "T1";"T2"}; {\ar @{..} "T2";"T3"}; {\ar @{..} "T3";"T0"}; % top
        {\ar @{..} "B0";"T0"}; {\ar @{..} "B1";"T1"}; {\ar @{..} "B2";"T2"}; {\ar @{..} "B3";"T3"}; % sides
        {\ar @{-} "B0";"B1"};  {\ar @{..} "B1";"B2"}; {\ar @{..} "B2";"B3"}; {\ar @{..} "B3";"B0"}; % bottom
        \end{xy}}="L00";
 (24,136)*+{\begin{xy}
        % vertices
          (4,12)*{}="T3"; (12,12)*{}="T2";
        (0,8)*{}="T0";  (8,8)*{\circ}="T1";
          (4,4)*{}="B3"; (12,4)*{\circ}="B2";
         (0,0)*{}="B0";  (8,0)*{}="B1";
        % edges
        {\ar @{..} "T0";"T1"}; {\ar @{..} "T1";"T2"}; {\ar @{..} "T2";"T3"}; {\ar @{..} "T3";"T0"}; % top
        {\ar @{-}  "B0";"T0"}; {\ar @{-}  "B1";"T1"}; {\ar @{..} "B2";"T2"}; {\ar @{..} "B3";"T3"}; % sides
        {\ar @{-}  "B0";"B1"}; {\ar @{-}  "B1";"B2"}; {\ar @{..} "B2";"B3"}; {\ar @{-}  "B3";"B0"}; % bottom
        \end{xy}}="L10";
 (48,128)*+{\begin{xy}
        % vertices
          (4,12)*{}="T3"; (12,12)*{}="T2";
        (0,8)*{}="T0";  (8,8)*{}="T1";
          (4,4)*{}="B3"; (12,4)*{}="B2";
         (0,0)*{}="B0";  (8,0)*{}="B1";
        % edges
        {\ar @{..} "T0";"T1"}; {\ar @{-} "T1";"T2"}; {\ar @{..} "T2";"T3"}; {\ar @{-}  "T3";"T0"}; % top
        {\ar @{-}  "B0";"T0"}; {\ar @{-} "B1";"T1"}; {\ar @{..} "B2";"T2"}; {\ar @{..} "B3";"T3"}; % sides
        {\ar @{-}  "B0";"B1"}; {\ar @{-} "B1";"B2"}; {\ar @{..} "B2";"B3"}; {\ar @{-}  "B3";"B0"}; % bottom
        \end{xy}}="L20";
 (48,144)*+{\begin{xy}
        % vertices
          (4,12)*{}="T3"; (12,12)*{}="T2";
        (0,8)*{}="T0";  (8,8)*{}="T1";
          (4,4)*{}="B3"; (12,4)*{}="B2";
         (0,0)*{}="B0";  (8,0)*{}="B1";
        % edges
        {\ar @{..} "T0";"T1"};  {\ar @{..} "T1";"T2"}; {\ar @{..} "T2";"T3"}; {\ar @{..} "T3";"T0"}; % top
        {\ar @{-}  "B0";"T0"};  {\ar @{-}  "B1";"T1"}; {\ar @{-}  "B2";"T2"}; {\ar @{-}  "B3";"T3"}; % sides
        {\ar @{-}  "B0";"B1"};  {\ar @{-}  "B1";"B2"}; {\ar @{..} "B2";"B3"}; {\ar @{-}  "B3";"B0"}; % bottom
        \end{xy}}="L21";
 (24,160)*+{\begin{xy}
        % vertices
          (4,12)*{}="T3"; (12,12)*{}="T2";
        (0,8)*{}="T0";  (8,8)*{}="T1";
          (4,4)*{}="B3"; (12,4)*{\circ}="B2";
         (0,0)*{}="B0";  (8,0)*{}="B1";
        % edges
        {\ar @{..} "T0";"T1"}; {\ar @{..} "T1";"T2"}; {\ar @{..} "T2";"T3"}; {\ar @{..} "T3";"T0"}; % top
        {\ar @{..} "B0";"T0"}; {\ar @{..} "B1";"T1"}; {\ar @{..} "B2";"T2"}; {\ar @{..} "B3";"T3"}; % sides
        {\ar @{-}  "B0";"B1"}; {\ar @{-}  "B1";"B2"}; {\ar @{..} "B2";"B3"}; {\ar @{-}  "B3";"B0"}; % bottom
        \end{xy}}="L11";
 (48,160)*+{\begin{xy}
        % vertices
          (4,12)*{}="T3"; (12,12)*{\circ}="T2";
        (0,8)*{}="T0";  (8,8)*{}="T1";
          (4,4)*{}="B3"; (12,4)*{}="B2";
         (0,0)*{}="B0";  (8,0)*{}="B1";
        % edges
        {\ar @{..} "T0";"T1"}; {\ar @{..} "T1";"T2"}; {\ar @{..} "T2";"T3"}; {\ar @{..} "T3";"T0"}; % top
        {\ar @{..} "B0";"T0"}; {\ar @{..} "B1";"T1"}; {\ar @{-}  "B2";"T2"}; {\ar @{-}  "B3";"T3"}; % sides
        {\ar @{-}  "B0";"B1"}; {\ar @{-}  "B1";"B2"}; {\ar @{..} "B2";"B3"}; {\ar @{-}  "B3";"B0"}; % bottom
        \end{xy}}="L22";
 (72,160)*+{\begin{xy}
        % vertices
          (4,12)*{}="T3"; (12,12)*{}="T2";
        (0,8)*{}="T0";  (8,8)*{}="T1";
          (4,4)*{}="B3"; (12,4)*{}="B2";
         (0,0)*{}="B0";  (8,0)*{}="B1";
        % edges
        {\ar @{..} "T0";"T1"}; {\ar @{-}  "T1";"T2"}; {\ar @{..} "T2";"T3"}; {\ar @{-} "T3";"T0"}; % top
        {\ar @{..} "B0";"T0"}; {\ar @{..} "B1";"T1"}; {\ar @{-}  "B2";"T2"}; {\ar @{-} "B3";"T3"}; % sides
        {\ar @{-}  "B0";"B1"}; {\ar @{-}  "B1";"B2"}; {\ar @{..} "B2";"B3"}; {\ar @{-} "B3";"B0"}; % bottom
        \end{xy}}="L30";
 (24,176)*+{\begin{xy}
        % vertices
          (4,12)*{}="T3"; (12,12)*{}="T2";
        (0,8)*{}="T0";  (8,8)*{\circ}="T1";
          (4,4)*{}="B3"; (12,4)*{}="B2";
         (0,0)*{}="B0";  (8,0)*{}="B1";
        % edges
        {\ar @{..} "T0";"T1"}; {\ar @{..} "T1";"T2"}; {\ar @{..} "T2";"T3"}; {\ar @{..} "T3";"T0"}; % top
        {\ar @{-}  "B0";"T0"}; {\ar @{-}  "B1";"T1"}; {\ar @{..} "B2";"T2"}; {\ar @{..} "B3";"T3"}; % sides
        {\ar @{-}  "B0";"B1"}; {\ar @{..} "B1";"B2"}; {\ar @{..} "B2";"B3"}; {\ar @{..} "B3";"B0"}; % bottom
        \end{xy}}="L12";
 (48,176)*+{\begin{xy}
        % vertices
          (4,12)*{}="T3"; (12,12)*{\circ}="T2";
        (0,8)*{}="T0";  (8,8)*{}="T1";
          (4,4)*{}="B3"; (12,4)*{}="B2";
         (0,0)*{}="B0";  (8,0)*{}="B1";
        % edges
        {\ar @{..} "T0";"T1"}; {\ar @{-}  "T1";"T2"}; {\ar @{..} "T2";"T3"}; {\ar @{-}  "T3";"T0"}; % top
        {\ar @{-}  "B0";"T0"}; {\ar @{-}  "B1";"T1"}; {\ar @{..} "B2";"T2"}; {\ar @{..} "B3";"T3"}; % sides
        {\ar @{-}  "B0";"B1"}; {\ar @{..} "B1";"B2"}; {\ar @{..} "B2";"B3"}; {\ar @{..} "B3";"B0"}; % bottom
        \end{xy}}="L23";
 (72,176)*+{\begin{xy}
        % vertices
          (4,12)*{}="T3"; (12,12)*{}="T2";
        (0,8)*{}="T0";  (8,8)*{}="T1";
          (4,4)*{}="B3"; (12,4)*{}="B2";
         (0,0)*{}="B0";  (8,0)*{}="B1";
        % edges
        {\ar @{..} "T0";"T1"}; {\ar @{-}  "T1";"T2"}; {\ar @{..} "T2";"T3"}; {\ar @{-}  "T3";"T0"}; % top
        {\ar @{-}  "B0";"T0"}; {\ar @{-}  "B1";"T1"}; {\ar @{-}  "B2";"T2"}; {\ar @{-}  "B3";"T3"}; % sides
        {\ar @{-}  "B0";"B1"}; {\ar @{..} "B1";"B2"}; {\ar @{..} "B2";"B3"}; {\ar @{..} "B3";"B0"}; % bottom
        \end{xy}}="L31";
 {\ar @{-} "L00";"L10"}; {\ar @{-} "L00";"L11"}; {\ar @{-} "L00";"L12"};
 {\ar @{-} "L10";"L20"}; {\ar @{-} "L10";"L21"};
 {\ar @{-} "L11";"L22"};
 {\ar @{-} "L12";"L23"};
 {\ar @{-} "L22";"L30"};
 {\ar @{-} "L23";"L31"};
%==================================================================================
(-35,32)*{\text{\small Edge orbits for $\rho_0$}};
 (-32,48)*{\begin{xy}
        % vertices
          (8,24)*{}="T3"; (24,24)*{}="T2";
        (0,16)*{}="T0";  (16,16)*{}="T1";
          (8,8)*{}="B3"; (24,8)*{}="B2";
         (0,0)*{}="B0";  (16,0)*{}="B1";
        % edges
        {\ar @{-}|>>>>>{\,5\,} "T0";"T1"}; {\ar @{-}|4 "T1";"T2"}; {\ar @{-} "T2";"T3"}; {\ar @{-}|3 "T3";"T0"}; % top
        {\ar @{-}|>>>>>{\strut 2} "B0";"T0"}; {\ar @{-}|>>>>>{\strut 1} "B1";"T1"}; {\ar @{-}|<<<<<{\strut 3} "B2";"T2"}; {\ar @{..}|<<<<<{\strut 4} "B3";"T3"}; % sides
        {\ar @{-} "B0";"B1"};  {\ar @{-}|2  "B1";"B2"}; {\ar @{..}|>>>>>{\,5\,} "B2";"B3"}; {\ar @{..}|{\strut 1} "B3";"B0"}; % bottom
        \end{xy}};
 % Level 0 -----------------------------------------------------------------------------------
 (0,48)*+{\begin{xy}
        % vertices
          (4,12)*{}="T3"; (12,12)*{}="T2";
        (0,8)*{}="T0";  (8,8)*{}="T1";
          (4,4)*{}="B3"; (12,4)*{}="B2";
         (0,0)*{}="B0";  (8,0)*{\circ}="B1";
        % edges
        {\ar @{..} "T0";"T1"}; {\ar @{..} "T1";"T2"}; {\ar @{..} "T2";"T3"}; {\ar @{..} "T3";"T0"}; % top
        {\ar @{..} "B0";"T0"}; {\ar @{..} "B1";"T1"}; {\ar @{..} "B2";"T2"}; {\ar @{..} "B3";"T3"}; % sides
        {\ar @{-} "B0";"B1"};  {\ar @{..} "B1";"B2"}; {\ar @{..} "B2";"B3"}; {\ar @{..} "B3";"B0"}; % bottom
        \end{xy}}="L00";
 % Level 1 -----------------------------------------------------------------------------------
 (24,96)*+{\begin{xy}
        % vertices
          (4,12)*{}="T3"; (12,12)*{}="T2";
        (0,8)*{}="T0";  (8,8)*{\circ}="T1";
          (4,4)*{}="B3"; (12,4)*{}="B2";
         (0,0)*{}="B0";  (8,0)*{}="B1";
        % edges
        {\ar @{..} "T0";"T1"}; {\ar @{..} "T1";"T2"}; {\ar @{..} "T2";"T3"}; {\ar @{..} "T3";"T0"}; % top
        {\ar @{..} "B0";"T0"}; {\ar @{-} "B1";"T1"}; {\ar @{..} "B2";"T2"}; {\ar @{..} "B3";"T3"}; % sides
        {\ar @{-} "B0";"B1"};  {\ar @{..} "B1";"B2"}; {\ar @{..} "B2";"B3"}; {\ar @{-} "B3";"B0"}; % bottom
        \end{xy}}="L12";
 (24,48)*+{\begin{xy}
        % vertices
          (4,12)*{}="T3"; (12,12)*{}="T2";
        (0,8)*{}="T0";  (8,8)*{}="T1";
          (4,4)*{}="B3"; (12,4)*{\circ}="B2";
         (0,0)*{}="B0";  (8,0)*{}="B1";
        % edges
        {\ar @{..} "T0";"T1"}; {\ar @{..} "T1";"T2"}; {\ar @{..} "T2";"T3"}; {\ar @{..} "T3";"T0"}; % top
        {\ar @{-} "B0";"T0"}; {\ar @{..} "B1";"T1"}; {\ar @{..} "B2";"T2"}; {\ar @{..} "B3";"T3"}; % sides
        {\ar @{-} "B0";"B1"};  {\ar @{-} "B1";"B2"}; {\ar @{..} "B2";"B3"}; {\ar @{..} "B3";"B0"}; % bottom
        \end{xy}}="L11";
 (24,8)*+{\begin{xy}
        % vertices
          (4,12)*{}="T3"; (12,12)*{}="T2";
        (0,8)*{}="T0";  (8,8)*{\circ}="T1";
          (4,4)*{}="B3"; (12,4)*{\circ}="B2";
         (0,0)*{}="B0";  (8,0)*{}="B1";
        % edges
        {\ar @{..} "T0";"T1"}; {\ar @{..} "T1";"T2"}; {\ar @{..} "T2";"T3"}; {\ar @{..} "T3";"T0"}; % top
        {\ar @{-} "B0";"T0"}; {\ar @{-} "B1";"T1"}; {\ar @{..} "B2";"T2"}; {\ar @{..} "B3";"T3"}; % sides
        {\ar @{-} "B0";"B1"};  {\ar @{-} "B1";"B2"}; {\ar @{..} "B2";"B3"}; {\ar @{-} "B3";"B0"}; % bottom
        \end{xy}}="L10";
 % Level 2 -----------------------------------------------------------------------------------
 (48,112)*+{\begin{xy}
        % vertices
          (4,12)*{}="T3"; (12,12)*{\circ}="T2";
        (0,8)*{}="T0";  (8,8)*{}="T1";
          (4,4)*{}="B3"; (12,4)*{}="B2";
         (0,0)*{}="B0";  (8,0)*{}="B1";
        % edges
        {\ar @{..} "T0";"T1"}; {\ar @{-} "T1";"T2"}; {\ar @{..} "T2";"T3"}; {\ar @{..} "T3";"T0"}; % top
        {\ar @{..} "B0";"T0"}; {\ar @{-} "B1";"T1"}; {\ar @{..} "B2";"T2"}; {\ar @{-} "B3";"T3"}; % sides
        {\ar @{-} "B0";"B1"};  {\ar @{..} "B1";"B2"}; {\ar @{..} "B2";"B3"}; {\ar @{-} "B3";"B0"}; % bottom
        \end{xy}}="L27";
 (48,96)*+{\begin{xy}
        % vertices
          (4,12)*{}="T3"; (12,12)*{}="T2";
        (0,8)*{\circ}="T0";  (8,8)*{}="T1";
          (4,4)*{}="B3"; (12,4)*{}="B2";
         (0,0)*{}="B0";  (8,0)*{}="B1";
        % edges
        {\ar @{-} "T0";"T1"}; {\ar @{..} "T1";"T2"}; {\ar @{..} "T2";"T3"}; {\ar @{..} "T3";"T0"}; % top
        {\ar @{..} "B0";"T0"}; {\ar @{-} "B1";"T1"}; {\ar @{..} "B2";"T2"}; {\ar @{..} "B3";"T3"}; % sides
        {\ar @{-} "B0";"B1"};  {\ar @{..} "B1";"B2"}; {\ar @{-} "B2";"B3"}; {\ar @{-} "B3";"B0"}; % bottom
        \end{xy}}="L26";
 (48,80)*+{\begin{xy}
        % vertices
          (4,12)*{}="T3"; (12,12)*{}="T2";
        (0,8)*{}="T0";  (8,8)*{}="T1";
          (4,4)*{}="B3"; (12,4)*{}="B2";
         (0,0)*{}="B0";  (8,0)*{}="B1";
        % edges
        {\ar @{-} "T0";"T1"}; {\ar @{-} "T1";"T2"}; {\ar @{..} "T2";"T3"}; {\ar @{..} "T3";"T0"}; % top
        {\ar @{..} "B0";"T0"}; {\ar @{-} "B1";"T1"}; {\ar @{..} "B2";"T2"}; {\ar @{-} "B3";"T3"}; % sides
        {\ar @{-} "B0";"B1"};  {\ar @{..} "B1";"B2"}; {\ar @{-} "B2";"B3"}; {\ar @{-} "B3";"B0"}; % bottom
        \end{xy}}="L25";
 (48,64)*+{\begin{xy}
        % vertices
          (4,12)*{}="T3"; (12,12)*{\circ}="T2";
        (0,8)*{}="T0";  (8,8)*{}="T1";
          (4,4)*{}="B3"; (12,4)*{}="B2";
         (0,0)*{}="B0";  (8,0)*{}="B1";
        % edges
        {\ar @{..} "T0";"T1"}; {\ar @{..} "T1";"T2"}; {\ar @{..} "T2";"T3"}; {\ar @{-} "T3";"T0"}; % top
        {\ar @{-} "B0";"T0"}; {\ar @{..} "B1";"T1"}; {\ar @{-} "B2";"T2"}; {\ar @{..} "B3";"T3"}; % sides
        {\ar @{-} "B0";"B1"};  {\ar @{-} "B1";"B2"}; {\ar @{..} "B2";"B3"}; {\ar @{..} "B3";"B0"}; % bottom
        \end{xy}}="L24";
 (48,48)*+{\begin{xy}
        % vertices
          (4,12)*{}="T3"; (12,12)*{}="T2";
        (0,8)*{}="T0";  (8,8)*{}="T1";
          (4,4)*{\circ}="B3"; (12,4)*{}="B2";
         (0,0)*{}="B0";  (8,0)*{}="B1";
        % edges
        {\ar @{-} "T0";"T1"}; {\ar @{..} "T1";"T2"}; {\ar @{..} "T2";"T3"}; {\ar @{..} "T3";"T0"}; % top
        {\ar @{-} "B0";"T0"}; {\ar @{..} "B1";"T1"}; {\ar @{..} "B2";"T2"}; {\ar @{..} "B3";"T3"}; % sides
        {\ar @{-} "B0";"B1"};  {\ar @{-} "B1";"B2"}; {\ar @{-} "B2";"B3"}; {\ar @{..} "B3";"B0"}; % bottom
        \end{xy}}="L23";
 (48,32)*+{\begin{xy}
        % vertices
          (4,12)*{}="T3"; (12,12)*{}="T2";
        (0,8)*{}="T0";  (8,8)*{}="T1";
          (4,4)*{}="B3"; (12,4)*{}="B2";
         (0,0)*{}="B0";  (8,0)*{}="B1";
        % edges
        {\ar @{-} "T0";"T1"}; {\ar @{..} "T1";"T2"}; {\ar @{..} "T2";"T3"}; {\ar @{-} "T3";"T0"}; % top
        {\ar @{-} "B0";"T0"}; {\ar @{..} "B1";"T1"}; {\ar @{-} "B2";"T2"}; {\ar @{..} "B3";"T3"}; % sides
        {\ar @{-} "B0";"B1"};  {\ar @{-} "B1";"B2"}; {\ar @{-} "B2";"B3"}; {\ar @{..} "B3";"B0"}; % bottom
        \end{xy}}="L22";
 (48,16)*+{\begin{xy}
        % vertices
          (4,12)*{}="T3"; (12,12)*{}="T2";
        (0,8)*{}="T0";  (8,8)*{}="T1";
          (4,4)*{}="B3"; (12,4)*{}="B2";
         (0,0)*{}="B0";  (8,0)*{}="B1";
        % edges
        {\ar @{..} "T0";"T1"}; {\ar @{..} "T1";"T2"}; {\ar @{..} "T2";"T3"}; {\ar @{-} "T3";"T0"}; % top
        {\ar @{-} "B0";"T0"}; {\ar @{-} "B1";"T1"}; {\ar @{-} "B2";"T2"}; {\ar @{..} "B3";"T3"}; % sides
        {\ar @{-} "B0";"B1"};  {\ar @{-} "B1";"B2"}; {\ar @{..} "B2";"B3"}; {\ar @{-} "B3";"B0"}; % bottom
        \end{xy}}="L21";
 (48,0)*+{\begin{xy}
        % vertices
          (4,12)*{}="T3"; (12,12)*{}="T2";
        (0,8)*{}="T0";  (8,8)*{}="T1";
          (4,4)*{}="B3"; (12,4)*{}="B2";
         (0,0)*{}="B0";  (8,0)*{}="B1";
        % edges
        {\ar @{..} "T0";"T1"}; {\ar @{-} "T1";"T2"}; {\ar @{..} "T2";"T3"}; {\ar @{..} "T3";"T0"}; % top
        {\ar @{-} "B0";"T0"}; {\ar @{-} "B1";"T1"}; {\ar @{..} "B2";"T2"}; {\ar @{-} "B3";"T3"}; % sides
        {\ar @{-} "B0";"B1"};  {\ar @{-} "B1";"B2"}; {\ar @{..} "B2";"B3"}; {\ar @{-} "B3";"B0"}; % bottom
        \end{xy}}="L20";
 % Level 3 -----------------------------------------------------------------------------------
 (72,112)*+{\begin{xy}
        % vertices
          (4,12)*{}="T3"; (12,12)*{}="T2";
        (0,8)*{}="T0";  (8,8)*{}="T1";
          (4,4)*{}="B3"; (12,4)*{}="B2";
         (0,0)*{}="B0";  (8,0)*{}="B1";
        % edges
        {\ar @{..} "T0";"T1"}; {\ar @{-} "T1";"T2"}; {\ar @{..} "T2";"T3"}; {\ar @{-} "T3";"T0"}; % top
        {\ar @{..} "B0";"T0"}; {\ar @{-} "B1";"T1"}; {\ar @{-} "B2";"T2"}; {\ar @{-} "B3";"T3"}; % sides
        {\ar @{-} "B0";"B1"};  {\ar @{..} "B1";"B2"}; {\ar @{..} "B2";"B3"}; {\ar @{-} "B3";"B0"}; % bottom
        \end{xy}}="L33";
 (72,96)*+{\begin{xy}
        % vertices
          (4,12)*{}="T3"; (12,12)*{}="T2";
        (0,8)*{}="T0";  (8,8)*{}="T1";
          (4,4)*{}="B3"; (12,4)*{}="B2";
         (0,0)*{}="B0";  (8,0)*{}="B1";
        % edges
        {\ar @{-} "T0";"T1"}; {\ar @{..} "T1";"T2"}; {\ar @{..} "T2";"T3"}; {\ar @{-} "T3";"T0"}; % top
        {\ar @{..} "B0";"T0"}; {\ar @{-} "B1";"T1"}; {\ar @{-} "B2";"T2"}; {\ar @{..} "B3";"T3"}; % sides
        {\ar @{-} "B0";"B1"};  {\ar @{..} "B1";"B2"}; {\ar @{-} "B2";"B3"}; {\ar @{-} "B3";"B0"}; % bottom
        \end{xy}}="L32";
  (72,64)*+{\begin{xy}
        % vertices
          (4,12)*{}="T3"; (12,12)*{}="T2";
        (0,8)*{}="T0";  (8,8)*{}="T1";
          (4,4)*{}="B3"; (12,4)*{}="B2";
         (0,0)*{}="B0";  (8,0)*{}="B1";
        % edges
        {\ar @{..} "T0";"T1"}; {\ar @{-} "T1";"T2"}; {\ar @{..} "T2";"T3"}; {\ar @{-} "T3";"T0"}; % top
        {\ar @{-} "B0";"T0"}; {\ar @{..} "B1";"T1"}; {\ar @{-} "B2";"T2"}; {\ar @{-} "B3";"T3"}; % sides
        {\ar @{-} "B0";"B1"};  {\ar @{-} "B1";"B2"}; {\ar @{..} "B2";"B3"}; {\ar @{..} "B3";"B0"}; % bottom
        \end{xy}}="L31";
 (72,48)*+{\begin{xy}
        % vertices
          (4,12)*{}="T3"; (12,12)*{}="T2";
        (0,8)*{}="T0";  (8,8)*{}="T1";
          (4,4)*{}="B3"; (12,4)*{}="B2";
         (0,0)*{}="B0";  (8,0)*{}="B1";
        % edges
        {\ar @{-} "T0";"T1"}; {\ar @{-} "T1";"T2"}; {\ar @{..} "T2";"T3"}; {\ar @{..} "T3";"T0"}; % top
        {\ar @{-} "B0";"T0"}; {\ar @{..} "B1";"T1"}; {\ar @{..} "B2";"T2"}; {\ar @{-} "B3";"T3"}; % sides
        {\ar @{-} "B0";"B1"};  {\ar @{-} "B1";"B2"}; {\ar @{-} "B2";"B3"}; {\ar @{..} "B3";"B0"}; % bottom
        \end{xy}}="L30";
 {\ar @{-} "L00";"L10"}; {\ar @{-} "L00";"L11"}; {\ar @{-} "L00";"L12"};
 {\ar @{-} "L10";"L20"}; {\ar @{-} "L10";"L21"};
 {\ar @{-} "L11";"L22"}; {\ar @{-} "L11";"L23"}; {\ar @{-} "L11";"L24"};
 {\ar @{-} "L12";"L25"}; {\ar @{-} "L12";"L26"}; {\ar @{-} "L12";"L27"};
 {\ar @{-} "L23";"L30"}; {\ar @{-} "L24";"L31"}; {\ar @{-} "L26";"L32"}; {\ar @{-} "L27";"L33"};
\end{xy}\]
\caption{\label{f:trees1}Enumerating  spanning trees with a given fixed edge that are fixed by $\varphi_0$ and $\rho_0$ }  % Fill in name of label and caption under figure.
\end{figure}
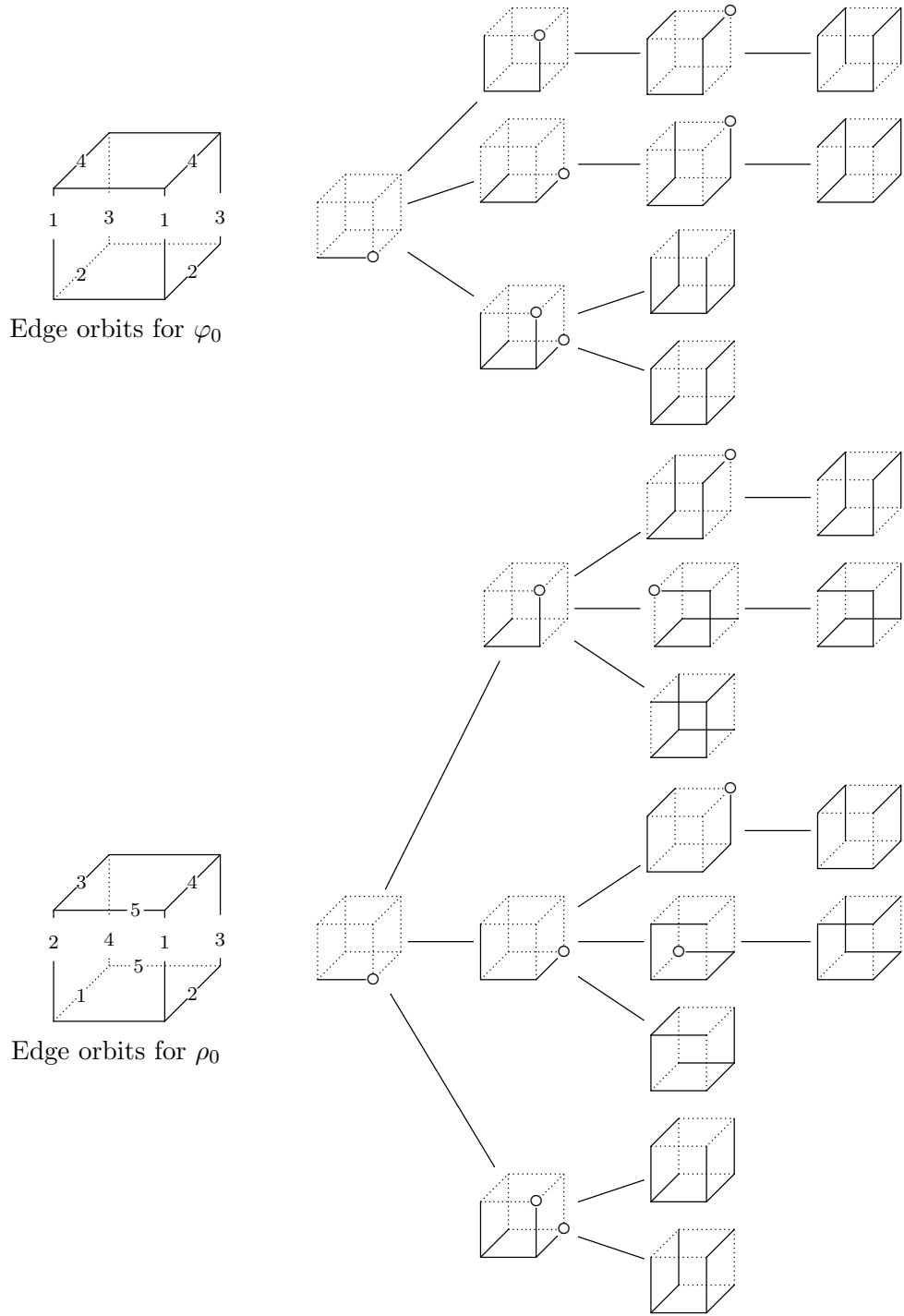

Recalling that the count for~$\rho_0$ must be doubled and the count
for~$\varphi_0$ must be quadrupled, we find that~$\fix(\rho_0)= 8
\times 2 =16$ and~$\fix(\varphi_0)= 4 \times 4 = 16$. Substituting
these values into Equation~\eqref{e:blfinal}, we get
\[
\text{\# orbits in~$\Upsigma$}=8 + \frac{16}{8} + \frac{16}{16}=11,
\]
thereby establishing

\begin{prop} \label{p:nunf}
The spanning trees of~$\cube$ are partitioned into eleven orbits under the action of~$\sym(\cube)$.
\end{prop}

The arguments given above do not exclude the possibility that there are fewer than eleven unfoldings of~$\cube$. It is conceivable that two spanning trees that are not equivalent under the action of~$\sym(\cube)$ nonetheless somehow give rise to unfoldings that are congruent in the plane. And it is also conceivable that spanning trees from distinct orbits are somehow isomorphic via a mapping that does not extend to a cube isometry.  So we now identify eleven incongruent shapes (Figure~\ref{11unfoldings}), thereby attaining the upper bound calculated in Proposition~\ref{p:nunf}, and so proving the following result:

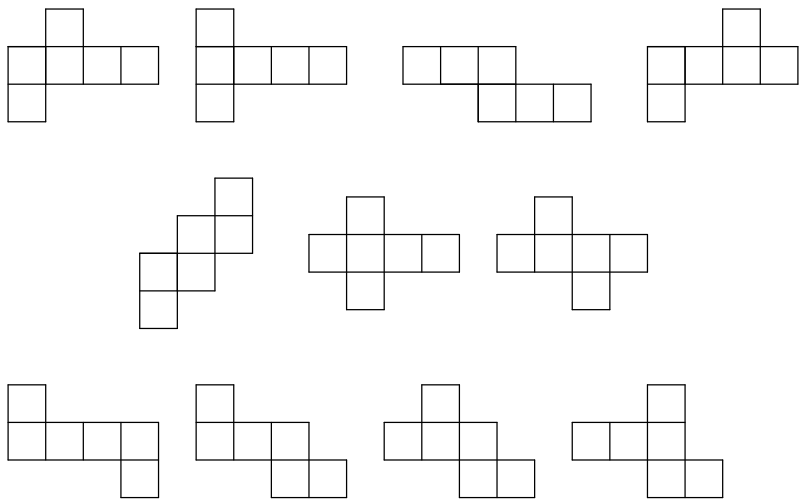
\begin{figure}[H]    % immediate placement first, else top of next page.
\[
 \begin{xy}
 (0,5)*{\begin{xy}
       (0,0)*{}="A"; (5,0)*{}="B"; (5,5)*{}="C";  (10,5)*{}="D"; (15,5)*{}="E"; (20,5)*{}="F";
       (20,10)*{}="G"; (15,10)*{}="H"; (10,10)*{}="I"; (10,15)*{}="J"; (5,15)*{}="K"; (5,10)*{}="L";
       (0,10)*{}="M"; (0,5)*{}="N";
       {\ar @{-} "A" ; "B"}; {\ar @{-} "B" ; "C"}; {\ar @{-} "N" ; "F"}; {\ar @{-} "F" ; "G"}; {\ar @{-} "G" ; "M"};
       {\ar @{-} "D" ; "J"}; {\ar @{-} "J" ; "K"}; {\ar @{-} "K" ; "C"}; {\ar @{-} "L" ; "M"}; {\ar @{-} "M" ; "A"};
       {\ar @{-} "L" ; "C"}; {\ar @{-} "E" ; "H"};
 \end{xy}};
 (25,5)*{\begin{xy}
       (0,0)*{}="A"; (5,0)*{}="B"; (5,5)*{}="C";  (10,5)*{}="D";
       (15,5)*{}="E"; (20,5)*{}="F";
       (20,10)*{}="G"; (15,10)*{}="H"; (10,10)*{}="I"; (0,15)*{}="J"; (5,15)*{}="K"; (5,10)*{}="L";
       (0,10)*{}="M"; (0,5)*{}="N";
       {\ar @{-} "A" ; "B"}; {\ar @{-} "B" ; "C"}; {\ar @{-} "N" ; "F"}; {\ar @{-} "F" ; "G"}; {\ar @{-} "G" ; "M"};
       {\ar @{-} "M" ; "J"}; {\ar @{-} "J" ; "K"}; {\ar @{-} "K" ; "C"}; {\ar @{-} "L" ; "M"}; {\ar @{-} "M" ; "A"};
       {\ar @{-} "L" ; "C"}; {\ar @{-} "E" ; "H"};{\ar @{-} "I" ; "D"};
 \end{xy}};
(55,0)*{\begin{xy}
       (-5,15)*{}="Q"; (0,15)*{}="R"; (-5,10)*{}="S";  (10,5)*{}="D"; (15,5)*{}="E"; (20,5)*{}="F";
       (20,10)*{}="G"; (15,10)*{}="H"; (10,10)*{}="I"; (10,15)*{}="J"; (5,15)*{}="K"; (5,10)*{}="L";
       (0,10)*{}="M";  (5,5)*{}="C";
       {\ar @{-} "Q" ; "J"};  {\ar @{-} "C" ; "F"}; {\ar @{-} "F" ; "G"}; {\ar @{-} "G" ; "M"};
       {\ar @{-} "D" ; "J"}; {\ar @{-} "S" ; "I"}; {\ar @{-} "K" ; "C"}; {\ar @{-} "R" ; "M"};
       {\ar @{-} "L" ; "C"}; {\ar @{-} "E" ; "H"}; {\ar @{-} "Q" ; "S"};
 \end{xy}};
(85,5)*{\begin{xy}
       (0,0)*{}="A"; (5,0)*{}="B"; (5,5)*{}="C";  (10,5)*{}="D"; (15,5)*{}="E"; (20,5)*{}="F";
       (20,10)*{}="G"; (15,10)*{}="H"; (10,10)*{}="I"; (10,15)*{}="J"; (15,15)*{}="K"; (5,10)*{}="L";
       (0,10)*{}="M"; (0,5)*{}="N";
       {\ar @{-} "A" ; "B"}; {\ar @{-} "B" ; "C"}; {\ar @{-} "N" ; "F"}; {\ar @{-} "F" ; "G"}; {\ar @{-} "G" ; "M"};
       {\ar @{-} "D" ; "J"}; {\ar @{-} "J" ; "K"}; {\ar @{-} "L" ; "C"}; {\ar @{-} "L" ; "M"}; {\ar @{-} "M" ; "A"};
       {\ar @{-} "L" ; "C"}; {\ar @{-} "E" ; "H"}; {\ar @{-} "K" ; "H"};
\end{xy}};
(15,-20)*{\begin{xy}
       (0,0)*{}="A"; (5,0)*{}="B"; (5,5)*{}="C";  (10,5)*{}="D"; (15,20)*{}="S"; (10,20)*{}="T";
       (15,15)*{}="R"; (15,10)*{}="Q"; (10,10)*{}="I"; (10,15)*{}="J"; (5,15)*{}="K"; (5,10)*{}="L";
       (0,10)*{}="M"; (0,5)*{}="N";
       {\ar @{-} "A" ; "B"}; {\ar @{-} "B" ; "C"}; {\ar @{-} "N" ; "D"};  {\ar @{-} "Q" ; "M"};
       {\ar @{-} "D" ; "T"}; {\ar @{-} "R" ; "K"}; {\ar @{-} "K" ; "C"}; {\ar @{-} "L" ; "M"}; {\ar @{-} "M" ; "A"};
       {\ar @{-} "L" ; "C"}; {\ar @{-} "Q" ; "S"}; {\ar @{-} "T" ; "S"};
\end{xy}};
(40,-20)*{\begin{xy}
       (0,15)*{}="D0"; (5,15)*{}="D1"; (10,15)*{}="D2"; (15,15)*{}="D3"; (20,15)*{}="D4";
       (0,10)*{}="C0"; (5,10)*{}="C1"; (10,10)*{}="C2"; (15,10)*{}="C3"; (20,10)*{}="C4";
        (0,5)*{}="B0";  (5,5)*{}="B1";  (10,5)*{}="B2";  (15,5)*{}="B3";  (20,5)*{}="B4";
        (0,0)*{}="A0";  (5,0)*{}="A1";  (10,0)*{}="A2";  (15,0)*{}="A3";  (20,0)*{}="A4";
       {\ar @{-} "A1";"A2"};{\ar @{-} "B0";"B4"};{\ar @{-} "C0";"C4"};{\ar @{-} "D1";"D2"};                      % Horizontals
       {\ar @{-} "B0";"C0"};{\ar @{-} "A1";"D1"};{\ar @{-} "A2";"D2"};{\ar @{-} "B3";"C3"};{\ar @{-} "B4";"C4"}; % Verticals
\end{xy}};
(65,-20)*{\begin{xy}
       (0,15)*{}="D0"; (5,15)*{}="D1"; (10,15)*{}="D2"; (15,15)*{}="D3"; (20,15)*{}="D4";
       (0,10)*{}="C0"; (5,10)*{}="C1"; (10,10)*{}="C2"; (15,10)*{}="C3"; (20,10)*{}="C4";
        (0,5)*{}="B0";  (5,5)*{}="B1";  (10,5)*{}="B2";  (15,5)*{}="B3";  (20,5)*{}="B4";
        (0,0)*{}="A0";  (5,0)*{}="A1";  (10,0)*{}="A2";  (15,0)*{}="A3";  (20,0)*{}="A4";
       {\ar @{-} "A2";"A3"};{\ar @{-} "B0";"B4"};{\ar @{-} "C0";"C4"};{\ar @{-} "D1";"D2"};                      % Horizontals
       {\ar @{-} "B0";"C0"};{\ar @{-} "B1";"D1"};{\ar @{-} "A2";"D2"};{\ar @{-} "A3";"C3"};{\ar @{-} "B4";"C4"}; % Verticals
\end{xy}};
(0,-45)*{\begin{xy}
       (0,15)*{}="D0"; (5,15)*{}="D1"; (10,15)*{}="D2"; (15,15)*{}="D3"; (20,15)*{}="D4";
       (0,10)*{}="C0"; (5,10)*{}="C1"; (10,10)*{}="C2"; (15,10)*{}="C3"; (20,10)*{}="C4";
        (0,5)*{}="B0";  (5,5)*{}="B1";  (10,5)*{}="B2";  (15,5)*{}="B3";  (20,5)*{}="B4";
        (0,0)*{}="A0";  (5,0)*{}="A1";  (10,0)*{}="A2";  (15,0)*{}="A3";  (20,0)*{}="A4";
       {\ar @{-} "A3";"A4"};{\ar @{-} "B0";"B4"};{\ar @{-} "C0";"C4"};{\ar @{-} "D0";"D1"};                      % Horizontals
       {\ar @{-} "B0";"D0"};{\ar @{-} "B1";"D1"};{\ar @{-} "B2";"C2"};{\ar @{-} "A3";"C3"};{\ar @{-} "A4";"C4"}; % Verticals
\end{xy}};
(25,-45)*{\begin{xy}
       (0,15)*{}="D0"; (5,15)*{}="D1"; (10,15)*{}="D2"; (15,15)*{}="D3"; (20,15)*{}="D4";
       (0,10)*{}="C0"; (5,10)*{}="C1"; (10,10)*{}="C2"; (15,10)*{}="C3"; (20,10)*{}="C4";
        (0,5)*{}="B0";  (5,5)*{}="B1";  (10,5)*{}="B2";  (15,5)*{}="B3";  (20,5)*{}="B4";
        (0,0)*{}="A0";  (5,0)*{}="A1";  (10,0)*{}="A2";  (15,0)*{}="A3";  (20,0)*{}="A4";
       {\ar @{-} "A2";"A4"};{\ar @{-} "B0";"B4"};{\ar @{-} "C0";"C3"};{\ar @{-} "D0";"D1"};                      % Horizontals
       {\ar @{-} "B0";"D0"};{\ar @{-} "B1";"D1"};{\ar @{-} "A2";"C2"};{\ar @{-} "A3";"C3"};{\ar @{-} "A4";"B4"}; % Verticals
\end{xy}};
(50,-45)*{\begin{xy}
       (0,15)*{}="D0"; (5,15)*{}="D1"; (10,15)*{}="D2"; (15,15)*{}="D3"; (20,15)*{}="D4";
       (0,10)*{}="C0"; (5,10)*{}="C1"; (10,10)*{}="C2"; (15,10)*{}="C3"; (20,10)*{}="C4";
        (0,5)*{}="B0";  (5,5)*{}="B1";  (10,5)*{}="B2";  (15,5)*{}="B3";  (20,5)*{}="B4";
        (0,0)*{}="A0";  (5,0)*{}="A1";  (10,0)*{}="A2";  (15,0)*{}="A3";  (20,0)*{}="A4";
       {\ar @{-} "A2";"A4"};{\ar @{-} "B0";"B4"};{\ar @{-} "C0";"C3"};{\ar @{-} "D1";"D2"};                      % Horizontals
       {\ar @{-} "B0";"C0"};{\ar @{-} "B1";"D1"};{\ar @{-} "A2";"D2"};{\ar @{-} "A3";"C3"};{\ar @{-} "A4";"B4"}; % Verticals
\end{xy}};
(75, -45)*{\begin{xy}
       (0,15)*{}="D0"; (5,15)*{}="D1"; (10,15)*{}="D2"; (15,15)*{}="D3"; (20,15)*{}="D4";
       (0,10)*{}="C0"; (5,10)*{}="C1"; (10,10)*{}="C2"; (15,10)*{}="C3"; (20,10)*{}="C4";
        (0,5)*{}="B0";  (5,5)*{}="B1";  (10,5)*{}="B2";  (15,5)*{}="B3";  (20,5)*{}="B4";
        (0,0)*{}="A0";  (5,0)*{}="A1";  (10,0)*{}="A2";  (15,0)*{}="A3";  (20,0)*{}="A4";
       {\ar @{-} "A2";"A4"};{\ar @{-} "B0";"B4"};{\ar @{-} "C0";"C3"};{\ar @{-} "D2";"D3"};                      % Horizontals
       {\ar @{-} "B0";"C0"};{\ar @{-} "B1";"C1"};{\ar @{-} "A2";"D2"};{\ar @{-} "A3";"D3"};{\ar @{-} "A4";"B4"}; % Verticals
\end{xy}};
\end{xy}
\]
\caption{The Eleven Unfolding Shapes of the Cube}
\label{11unfoldings}
\end{figure}

\begin{thm}
There are exactly eleven incongruent unfoldings of the cube.
\end{thm} 

We haven't forgotten the original problem of identifying shortest paths on the cube.  We tackle this question in \cite{goldsuzz},  where the number of unfoldings plays a role.  We also find a direct combinatorial way to enumerate unfoldings, which is in fact how the entries in Figure~\ref{11unfoldings} were obtained.

\subsection*{Looking back}

The mystery of the unfolding shapes of~$\cube$ has been resolved, and looking back we can discern the following path. Unfoldings corresponded to spanning trees of the cube graph---counted via the Matrix-Tree theorem.  Counting the incongruent unfolded shapes required Burnside's lemma, which in turn required counting spanning trees invariant under cube isometries. Algebra said that each element of an action of a subgroup of~$S_4 \times \zmod{2}$ on a seven-element set has to have fixed points, so invariant spanning trees had to have at least one invariant edge.  Geometry identified which isometries could have such an invariant edge, and graph theory winnowed the candidates and indicated that at most one edge could be invariant in any one invariant spanning tree.  The upshot was that the invariant spanning trees of only two cube isometries had to be analyzed by hand, which could be done systematically using the symmetry principles learned along the way.

%=============================================================================
\end{document}